\documentclass[letter,12pt]{article}
\usepackage[T1]{fontenc}
\usepackage{verbatim,comment}
\usepackage{amsmath,amsthm}
\usepackage{xspace}
\usepackage{pifont}
\usepackage{graphicx}
\usepackage{amssymb}
\usepackage{epic, eepic}
\usepackage{dsfont}
\usepackage{amssymb}
\usepackage{makeidx}
\usepackage{mathrsfs,enumerate}

\usepackage{amsmath,afterpage}
\usepackage{epsf}
\usepackage{graphics,color}

\newcommand{\zz}[1]{\mathbb{#1}}
\def\q{\quad}
\def\qq{{\qquad}}

\def\0{\mathbf{0}}

\def\eps{\varepsilon}

\def \th{\theta}
\def\rr{\rightarrow}
\def\dr{\downarrow}

\def \< {\langle}
\def \> {\rangle}

\def\ol{\overline}

\def\beqa{\begin{eqnarray}}
\def\eeqa{\end{eqnarray}}
\def\beqas{\begin{eqnarray*}}
\def\eeqas{\end{eqnarray*}}

\newtheorem{theorem}{Theorem}[section]
\newtheorem{lemma}[theorem]{Lemma}
\newtheorem{proposition}[theorem]{Proposition}
\newtheorem{prop}[theorem]{Proposition}

\newtheorem{corollary}[theorem]{Corollary}

\newtheorem{remark}[theorem]{Remark}

\numberwithin{equation}{section}
\newcommand{\hatd}[1]{{}}

\newcommand{\bd}{\begin{displaymath}}
\newcommand{\ed}{\end{displaymath}}
\newcommand{\be}{\begin{equation}}
\newcommand{\ee}{\end{equation}}
\newcommand{\bq}{\begin{eqnarray}}
\newcommand{\eq}{\end{eqnarray}}
\newcommand{\bn}{\begin{eqnarray*}}
\newcommand{\en}{\end{eqnarray*}}
\newcommand{\dl}{\delta}
\newcommand{\re}{\zz{R}}
\newcommand{\ze}{\zz{Z}}
\newcommand{\W}{\mathcal{W}}

\def\wt{\widetilde}

\newcommand{\indic}{\mathds1}

\def\n{N}				       	   

\newcommand\arcoth{\mbox{arcoth}}				       	   

\begin{document}
\title{On the Maximal Displacement of Near-critical Branching Random Walks\footnote{Research partially supported by GRF 606010 and 607013 of the HKSAR and the HKUST IAS Postdoctoral Fellowship.} }
\author{Eyal Neuman\footnote{Eyal Neuman would like to thank HKUST where part of the research was carried out.}, Xinghua Zheng}

\date{\today}

\maketitle

\abstract{We consider a branching random walk on $\mathbb{Z}$ started by $n$ particles at the origin, where each particle disperses
according to a mean-zero random walk with bounded support and reproduces with mean number of offspring
$1+\theta/n$. For $t\geq 0$, we study $M_{nt}$, the rightmost position reached by the branching random walk up to generation~$[nt]$.
Under certain moment assumptions on the branching law, we prove that $M_{nt}/\sqrt{n}$ converges weakly to the rightmost support point of the local time of the limiting super-Brownian motion. The convergence result establishes a sharp exponential decay of the tail distribution of $M_{nt}$. We also confirm that when $\th>0$, the support of the branching random walk grows in a linear speed that is identical to that of the limiting super-Brownian motion which was studied by Pinsky in \cite{pinsky95}. The rightmost position over all generations, $M:=\sup_t M_{nt}$, is also shown to converge weakly to that of the limiting super-Brownian motion, whose tail is found to decay like a Gumbel distribution when $\th<0$.}

\section{Introduction and Main Results}\label{sec-intro}
The study of extreme values of branching particle systems has attracted a considerable amount of attention during the last few decades. Early works on the tail behavior of branching Brownian motion trace back to Sawyer and Fleischman \cite{S-F78} and Lalley and Sellke \cite{Lalley-Sellke1987}.
During the same time period, the strong law of large numbers for the maxima of branching random walk was studied by Hammersley \cite{Hammersley}, Kingman \cite{Kingman}, Biggins \cite{Biggins} and Bramson \cite{Bramson78}.

The tail behavior of the maximal displacement of branching random walk was only derived recently. We classify these results into three subclasses according to the mean number of offspring, which we denote by $\mu$.  In the supercritical case ($\mu>1$),
the asymptotic tail distribution of the position of the rightmost particle was derived by Aidekon in \cite{Aidekon2013}. It was proved by Aidekon that
the maximal displacement  converges weakly to a random  shift of the Gumbel distribution
(see also 
\cite{Bachmann:2000, Hu-Shi09, AR09,Bramson-Z09,BDZ16}).

The subcritical case ($\mu<1$) was studied in \cite{NZ17}. It was proved in \cite{NZ17} that {the} tail distribution of the position of the rightmost particle decays exponentially. Moreover, the exact rate of decay was derived.

The case where the branching law is critical, that is $\mu=1$, was studied by Lalley and Shao in \cite{LS15}. Let $R_{n}$ be the rightmost position at generation $n$ of a branching random walk started by one particle at the origin. It was proved in \cite{LS15} that when the jump distribution has mean $0$, then under some moment assumptions, the  distribution of $R_{n}/\sqrt{n}$ {conditional on the branching process surviving for $n$ generations}, converges weakly to a  distribution $G$ given by
\bn
G(x)=P_{\delta_0}\big(\wt X_{1}[x,\infty)=0\,|\, {\wt X_1(-\infty,\infty) > 0}\big),
\en
where $\{\wt X_{t}\}_{t \geq 0}$ is {a} super-Brownian motion, and $P_{\delta_0}$ stands for the probability distribution of $(\wt X_t)$ with $\wt X_0=\delta_0$,  the Dirac  measure at the origin.

In this paper, we consider the near-critical case, namely, when the mean number of offspring{s} $\mu=1+\th/n$ for some $\th\in\zz{R}$.
This is a regime where phase transitions occur and interesting phenomena arise. Moreover, different from \cite{LS15}, we consider the rightmost position of the \emph{local time} process rather than the process itself. The local time process plays a critical role in some other studies, for example, the study of the \emph{susceptible-infected-recovered (SIR)} epidemic model (\cite{L-Z2010,LPZ14}).

Specifically, let $X_k(x)$ denote the number of particles at site $x$ at generation~$k$. Recall that the local time process of a spatial  particle system $X$ is given by
\bn
L_{m}(x) = \sum_{k \leq m}X_{k}(x), \q\mbox{for all } m \in\zz{Z}_{\geq 0} \mbox{ and } x\in\zz{Z}.
\en
Suppose that the branching random walk starts with $n$ particles at the origin and the mean number of offsprings  is $1+\th/n$ for some $\th\in\zz{R}$. It is well known that if each particle carries mass $1/n$, and if we rescale time by $1/n$ and space by $1/\sqrt{n}$, then as $n\rr\infty$, the measure-valued process converges weakly to a super-Brownian motion with drift $\th$; see, e.g., \cite{Perk2002}. The  rescaled process
$\big(n^{-3/2}L_{[ nt]}(\sqrt{n}x)\big)_{t\geq 0,\, x\in\zz{R}}$ also converges weakly to the local time density of the super-Brownian motion; see \cite{Lalley2009, L-Z2010}.

Note that the maxima of the support of $L_{nt}$ equals $M_{nt}$, the rightmost position reached by the branching random walk up to generation $[nt]$. The weak convergence of the branching random walk to super-Brownian motion, however, does not imply the weak convergence of $M_{nt}/\sqrt{n}$. The reason is that $M_{nt}/\sqrt{n}$ is not a continuous function of measures with respect to the topology of weak convergence (see, for example, the discussion after Theorem~3 in \cite{LS15}). Our first main result, Theorem \ref{theorem-weak-con}, confirms that $M_{nt}/\sqrt{n}$ converges weakly to~$\wt M_t$, the rightmost support point of the limiting super-Brownian up to time~$t$.

\cite{LS15} also studied  the tail distribution of the rightmost position over all generations, namely,
\be \label{max-inf}
M = \sup_{k\geq 0} M_k.
\ee
It was proved in \cite{LS15} that in the critical case and under some moment assumptions,
\begin{equation}\label{eq:LS_1d}
P\big(M\geq x\big) \sim \frac{\alpha}{x^{2}}, \textrm{ as } x \rr\infty .
\end{equation}
Here $\alpha$ is a constant that depends on the standard deviations of the jump distribution and the offspring distribution.
The asymptotics \eqref{eq:LS_1d} implies that for a critical branching random walk started with $n$ particles at the origin, the tail distribution of $M/\sqrt{n}$, that is $P\big(M\geq \sqrt{n}x\big)$, decays at  rate $O(1/x^2)$ for large values of $n$ (see Corollary 2 in \cite{LS15}). We will show in Corollary \ref{corol-u-x}(i) that the corresponding tail distribution of $M_{nt}/\sqrt{n}$ decays with a rate of $\exp(- c(t) x^2)$. The difference between the two convergence rates implies that the heavy-tail behavior of  $M$ in the critical case is due to particles that survive more than~$O(n)$ generations.

In the supercritical case,
 precise estimates on the tail distribution of the radius of the support of a super-Brownian motion
were established by Pinsky {in \cite{pinsky-pde} and~\cite{pinsky95}}.
Let $B_{r}(0)$ be the ball of radius $r$ centered at the origin. It was proved in \cite{pinsky95} (see equation~(6) therein) that for a super-Brownian motion $\wt X=\{\wt X^{}_{t}\}_{t\geq 0}$  with drift $\theta>0$, diffusion coefficient $\sigma_R^2$  and branching coefficient~$\sigma^2$, one has
\bn
P_{\dl_{0}}^{\theta} \big(\wt L^{}_{t}(B_{r}(0)^{c}) =0 \big) = e^{-u_{r}{(t,0)}},
 \en
where $P^\theta_{\dl_{0}}$ stands for the probability distribution of $\wt X^{}$ with drift $\theta$, $\wt X^{}_0=\dl_0$, and
\[
\wt L_t=\int_0^t \wt X_s\,ds
\]
 is the local time process associated with $\wt X.$
As to $u_{r}{(\cdot)},$ for any $r>0$, $u_{r}({t,x})$ is the minimal {positive} solution to the following nonlinear PDE:
 \be
 \left\{
 \begin{aligned}  \label{sing-pde}
 \frac{\partial u}{\partial t} &= \frac{\sigma_R^{2}}{2} { \frac{\partial^{2}u}{\partial x^{2} }}  +\theta   u - \sigma^2 u^{2}, \quad {t >0, \ |x|<r}, \\
 u({0,x}) &= 0 \textrm{ on } |x|<r, \\
  \lim_{|x| \rr  r} u({t,x})& = \infty.
 \end{aligned}
 \right.
 \ee
The existence of a positive solution to (\ref{sing-pde}) was derived in Theorem 1 of \cite{pinsky-pde} along with some sharp bounds on the minimal positive solution. The uniqueness of  positive solutions to \eqref{sing-pde} can be established by a similar argument to the proof of Proposition~\ref{prop-uniq} in this paper.

One important implication of Theorem 1 of \cite{pinsky-pde} is  the growth rate of the support of $\wt X$. It was proved in Theorem 1 of \cite{pinsky95} that the large time growth rate  is linear with rate $\bar \gamma := (2\th \sigma_R ^{2})^{-1/2}$. Specifically, one has
$$
\lim_{t\rr \infty }P^\theta_{\dl_{0}} \big(\wt L_{\gamma t}(B_{t}(0)^{c}) =0  \big)  =1, \quad \textrm{if } \gamma <\bar \gamma.
$$
and
$$
\lim_{t\rr \infty }P^{\theta}_{\dl_{0}} \big(\wt X_{\gamma t}(B_{t}(0)^{c}) >0 \, | \, { \zeta_{\wt X} =\infty} \big)  =1, \quad \textrm{if } \gamma >\bar \gamma,
$$
{where $\zeta_{\wt X}$ is the extinction time of $\wt X$. }
The above convergence in probability is strengthened to be almost sure convergence in \cite{Kyprianou05}.

The aforementioned growth rate result brings up the second aim of this paper, namely,  to derive the growth rate of near-critical  branching random walks. As we mentioned earlier,
results on the limiting measure-valued processes in most cases are not precise enough for the research of discrete particle systems.
The motivation for this work comes from the study of population and epidemic models, where sharp bounds on the local time are key elements in the proofs of phase transitions. For example, in \cite{LPZ14}, a phase transition for the spatial measure-valued \emph{susceptible-infected-recovered (SIR)} epidemic models was established. A key ingredient in the proof  is the growth rate of the support of the local time (see the discussion in Section~1.2 of \cite{LPZ14}).

Before we state our main results, we define more carefully the branching random walk that we study.

\textbf{The model:} For any fixed $n\in \zz{\n}$ and a constant $\eta\in\zz{R}$, $P^{\eta}_{n}$ stands for the probability distribution of  a discrete time branching random walk
$X^{}=(X^{}_k(x))_{k\geq 0, x\in \zz{Z}}$ initiated by $n$ particles at the origin and
with the following properties. In each generation, particles first jump (independently from each other) according a distribution with a finite  range, $F_{RW}= \{a_{k}\}_{k\in [-R,R]}$, which has mean $0$ and variance $\sigma_R^{2}$, and then each particle branches independently according to an offspring distribution $F^{\eta}_{B}=\{p^{\eta}_{i}\}_{i\geq 0}$, which has $p^{\eta}_{0}>0$, expectation $1+\eta$, variance $\sigma^{2}(\eta)$ and third moment $\gamma(\eta)$. The $\sigma(\eta)$ and $\gamma(\eta)$ satisfy that {for some $\dl>0$,

\be \label{assum-sig-gam}
\lim_{\eta \dr 0} \sigma(\eta) = \sigma>0, \quad \sup_{\eta \in (0,\dl) }\gamma(\eta) < \infty.
\ee
We remark that (\ref{assum-sig-gam}) is the only assumption that we make on $F^{\eta}_{B}$ for different  values of $\eta$.

\paragraph{Notation.}
We often use the abbreviated notation $P_{n}= P_{n}^{0}$,  $F_{B}= F^{0}_{B}$, $P^\eta= P_{1}^{\eta}, E^\eta= E_{1}^{\eta}$, etc.

Observe that under this model, particles jump first and then reproduce, just as in \cite{LS15}.
This does not change the limiting tail behavior of the maximal displacement as explained in Remark~3 therein and noting the Taylor expansion of function $Q$ given by \eqref{Q-\n1} below.

We study the tail behavior of the maximal displacement of ${X}$ up to generation $[nt]$ for $t\geq 0$, that is,
\bq \label{M-n-def}
M_{nt} = \max\{z\in \ze\,|\, L_{nt}(z)>0 \}.
\eq
Define $(u_{k}^{\eta}(y))$ to be function obtained by linear interpolation in $y$ from the values
\bn
u_{k}^{\eta}(y)=P_{n}^{\eta} (M_{k}\geq y), \  y \in \zz{Z}, \ k \in\zz{Z}_{\geq 0}.
\en

Let $\th\in\zz{R}$ and $n\in \mathbb{\n}$. Our first main result establishes the weak convergence of $M_{nt}/\sqrt{n}$ under $P^{\th/n}_{n}$, to the rightmost point in the support of the local time~$\wt L_t$ of the limiting  super-Brownian motion.

\begin{theorem} \label{theorem-weak-con}
For every $t\geq 0$ and $x\geq 0$,
\bn
\lim_{n\rr\infty }u_{nt}^{\th/n}(\sqrt{n}x)=P^\theta_{\dl_{0}}\big(\wt L_{t}([x,\infty)) >0\big).
\en
\end{theorem}

\begin{remark} \label{remark-weak}
The convergence in Theorem \ref{theorem-weak-con} is new even when $\th=0$, i.e., the critical case. In Corollary 2 of \cite{LS15}, the tail behaviour of the maximal displacement over all generations was derived for the critical case. However, such results are very different from the tail behaviour of $M_{nt}$, which describes the propagation of the support of  branching random walks. In terms of {the} proof, Theorem \ref{theorem-weak-con} requires different methods for deriving the tightness of the sequence $\left(u^{
\th/n}_{[ n_{} t]}(\sqrt{n} x): t\geq 0, \, x>0\right)_{n\geq 1}$ as shown in Section \ref{sec-conv}, and for the analysis of the limiting functional which satisfies a singular parabolic PDE as discussed in Section \ref{section-pde-lim}.
\end{remark}

\begin{remark}
When $\th\neq 0$ and {$F^{\th/n}_{B}$} is {either Poisson or  Binomial with mean~$1+\theta/n$}, the convergence of $u_{n\cdot}^{\th/n}(\sqrt{n}\cdot)$ follows from the convergence in the critical case and the convergence of the likelihood ratio between the near-critical and critical systems. {We refer to Section 3.3 of \cite{L-Z2010} for a similar argument on the convergence of the likelihood ratio between SIR epidemics and critical branching random walks.}  In general, such an argument  fails because there may be no likelihood ratio not to mention its convergence.
\end{remark}

\begin{remark}
The finite range assumption of the underling random walk facilitates controlling the overshoot of the random walk in analyzing a discrete Feynman-Kac formula; see, for example, \eqref{eq:nw_bd_exit_prob}, \eqref{eq:I2}, \eqref{j-2-fr} and \eqref{gf1}. By using the finite moment results of the  overshoot distribution (Lemma 10 in \cite{LS15} and Exercise 6, p. 232 of \cite{Spitzer}), one can relax the assumption to be finite 5th moment (the exponential decay in Lemma \ref{tight-n-w} (ii) becomes polynomial decay, but main theorems remain true.) See Remark \ref{rmk:I_2_5th_moment} for an example on how to modify the argument under a finite 5th moment assumption.
\end{remark}

We now describe {a corollary} to Theorem \ref{theorem-weak-con}. Consider the following Fisher-Kolmogorov-Petrovskii-Piscounov (FKPP) equation
\be\label{kpp-uniq}
\left\{
\begin{aligned}
\frac{\partial \phi}{\partial t}(t,x) &=\frac{\sigma_R^2}{2} \frac{\partial^{2}\phi}{\partial x^{2} }(t,x)+\theta \phi(t,x) - \frac{\sigma^2}{2} \phi^{2}(t,x), \quad t>0, \ x>0, \\
 \lim_{t \dr 0}\phi (t,x) &= 0, \mbox{ for all } x>0\\
\lim_{x \dr 0} \phi(t,x) &= \infty, \mbox{ for all } t>0\\
\lim_{x\rr \infty}   \phi(t,x) &= 0 \quad \textrm{uniformly on } [0,T] \mbox{ for any } T>0.
\end{aligned}
\right.
\ee
The existence and uniqueness of  positive solutions to (\ref{kpp-uniq}) will be proved in Section \ref{section-pde-lim}.

We will also need the following nonlinear ODE. By Proposition 2 in \cite{pinsky-pde}, for each $\rho \in (0,\sqrt{2})$, there exists a unique positive increasing solution $f=f_\rho\in \mathcal C^2([0,\infty))$ to the equation
\be \label{trav-wave}
\left\{
\begin{aligned}
\frac{1}{2} f^{''} -\rho f' +f-f^2 & = 0, \quad x\geq 0 \\
 f(0) &=0\\
 \lim_{x\rr \infty} f(x) &=1.
\end{aligned}
\right.
\ee
Moreover, one has
\bn
\lim_{x\rr \infty} \frac{1}{x} \log(1-f_\rho(x)) = \rho - (\rho^2+2)^{1/2}.
\en
We will show in Corollary \ref{corol-u-x} that a traveling-wave sub-solution of \eqref{kpp-uniq} can be obtained from $f_{\rho}$. See Section 1 of \cite{HHK06} for a discussion about the FKPP equation and traveling-wave solutions.

In the following corollary, we derive some exponential bounds on $u_{n \cdot}^{\th/n}(\sqrt{n}\cdot)$.
\begin{corollary}  \label{corol-u-x}
For all $ x>0, t>0$, we have
\be \label{u-th-asymp}
\lim_{n\rr\infty }u_{nt}^{\th/n}(\sqrt{n}x)=  1- e^{-\phi(t,x)},
\ee
where the following bounds on $\phi$ hold:
 \begin{enumerate}[(i)]
\item  when $\th\geq 0$, for every $\eps>0$, there exists $c_{\eps}>0$ and $M_{\eps}>0$ such that for all $ x>M_{\eps}$,
$$
\phi(t,x) \leq \frac{1}{\sigma^2}\Big(\theta+\frac{12\sigma_R^{2}}{x^2}\Big)\exp\Big( -\Big(\frac{x^{2}}{2\sigma_R^{2}(1+\eps)t }- \theta t -c_{\eps}\Big)_{+}\Big), \quad \textrm{for all } t\geq 0.
 $$
\item when $\th> 0$, for each $\rho \in (0,\sqrt{2})$,  we have
\bn
\phi(t,x) \geq  \frac{2\theta}{\sigma^2} f_{\rho}\left( \left(\rho \theta t - \frac{\sqrt{\theta}}{\sigma_R}x\right)_+\right), \quad \textrm{for all } x>0, \  t\geq 0.
\en
\end{enumerate}
\end{corollary}

\begin{remark}
 The convergence  in \eqref{u-th-asymp} \emph{cannot be} strengthened to be uniform convergence in $(t,x)$ over any infinite domain of the form $\{(t,x): t\geq t_0,\ x\geq x_0\}$. The reason is due to the discontinuity of $u_{n\cdot}^{\th/n}(\sqrt{n}\cdot)$ near the boundary of the support of the branching random walk; see the next theorem for the precise statement.
\end{remark}

The next main result establishes the large time growth rate of the support of $X^{\theta/n}$:
\begin{theorem} \label{theorem-speed}
Let $\bar \gamma = (2\th \sigma_R^{2})^{-1}$ for $\theta>0$.
\begin{itemize}
\item[(i)] For any $\gamma <\bar \gamma$,  there exists $N(\gamma)>0$ such that for all $n>N(\gamma)$,
$$
P^{\theta/n}_{n}\big(L_{n\gamma t}((\sqrt{n} t,\infty)) =0 \mbox{ for all } t \mbox{ large enough}\big)  =1;
$$
\item[(ii)]  for any $\gamma >\bar \gamma$,  there exists $N(\gamma)>0$ such that for all $n>N(\gamma)$,
$$
P^{\theta/n}_{n} \big(X^{}_{n\gamma t }((\sqrt{n} t,\infty)) >0  \mbox{ for all } t \mbox{ large enough} \, | \,   X \mbox{ survives} \big)  =1.
$$
\end{itemize}
\end{theorem}

\begin{remark}\label{rmk:speed_BRW}
In \cite{pinsky95}, the author established the linear growth of the support of supercritical super-Brownian motions by observing that the bounds on $(\phi(t,x))$  given in Corollary \ref{corol-u-x} imply that
\begin{equation}\label{eq:lim_fn_gamma}
\lim_{t\to \infty} P^{\theta}_{\delta_0}\big(\wt L_{\gamma t}((t,\infty)) =0  \big)
=\left\{
\begin{aligned} &0 &\mbox{ if } \gamma  <\bar \gamma\\
         &P^{\theta}_{\delta_0}({\zeta_{\wt X}<\infty }) &\mbox{ if } \gamma  >\bar \gamma,
\end{aligned}
\right.
\end{equation}
{where $\zeta_{\wt X}$ is the extinction time of $\wt X$.}
It follows from Corollary {\ref{corol-u-x}} that for any $\eps>0$, there exist $T_0>0$ such that for any $t>T_0$, we can find an $N_0=N_0(t)$ such that for all $n\geq N_0$,
\begin{equation}\label{eq:n_fn_gamma}
P^{\theta/n}_{n}\big( L_{n \gamma t}((\sqrt{n} t,\infty)) =0  \big)
\left\{
\begin{aligned} &\leq \eps &\mbox{ if } \gamma  <\bar \gamma\\
         &\leq P^{\theta}_{\delta_0}({\zeta_{\wt X}<\infty} ) +\eps &\mbox{ if } \gamma  >\bar \gamma.
\end{aligned}
\right.
\end{equation}
This result, however, is not enough to establish the linear growth of the support of branching random walk because we would need \eqref{eq:n_fn_gamma} to hold for \emph{all} $t$ large enough.
Such a \emph{uniform} convergence seems difficult to prove given the discontinuity of the limit in \eqref{eq:lim_fn_gamma} as a function of $\gamma$. We prove the linear growth result using  another argument.

\end{remark}

Finally, {analogous to the critical case in} \cite{LS15}, we derive the  tail distribution of the maximal displacement over all generations, namely, $M=\sup_k M_k$.

\begin{theorem} \label{theorem-u-inf}
When $\th\neq 0$, for every $x_0> 0$, uniformly over $x\in [x_0,\infty),$ we have
\begin{equation}\label{eq:conv_M}
\lim_{n\rr\infty }P^{\theta/n}_n(M\geq \sqrt{n}x) =  1-e^{-\psi(x)},
\end{equation}
where
\begin{equation} \label{eq:psi_formula}
 \psi(x) = \frac{2\theta^{+}}{\sigma^2} + \frac{3|\theta| }{\sigma^{2}} \left( \coth^2\left(\sqrt{\frac{|\theta| }{2\sigma_R^2}}x\right)-1 \right), \quad x>0,
\end{equation}
 is the unique solution to
\begin{equation} \label{psi-eq}
\left\{
\begin{aligned}
\frac{\sigma^2_R}{2} \frac{\partial^{2}{\psi}}{\partial x^{2} }&= -\theta \sigma^2 \psi +\frac{\sigma^2}{2}\psi^2 ,  \quad x >0, \\
\lim_{x \to 0+} \psi(x) &= \infty,\\
\lim_{x \to \infty} \psi(x) =&\frac{2\th^+}{\sigma^2}.
\end{aligned}
\right.
\end{equation}
Moreover, we have
\begin{equation} \label{psi-asymptotic}
  \psi(x) \sim \frac{2\theta^+}{\sigma^2} + \frac{12 |\theta|}{\sigma^2} \exp\left(-\sqrt{\frac{2 |\theta| }{\sigma_R^2}}\, x\right),\q\mbox{as } x\to \infty,
\end{equation}
where $\theta^+=\max(\theta,0).$
\end{theorem}
\begin{remark}
When $\th=0$, the convergence \eqref{eq:conv_M} and ODE \eqref{psi-eq} are also true; see Corollary 2 and Proposition 23 (and its proof) in \cite{LS15}.
The tail behavior of~$M$, however, is completely different according to whether $\th=0$ or not. When~$\th=0,$ by Corollary 2 in \cite{LS15},
the tail distribution of $M/\sqrt{n}$ decays at a rate of~$1/x^2$. {In fact, by solving \eqref{psi-eq} with $\theta=0$ along the same lines as in the proof of Theorem \ref{theorem-u-inf}, one gets that $$\psi(x) =  \frac{6\sigma_R^2}{\sigma^2}\frac{1}{x^{2}}, \quad x>0.$$ A theorem by Dynkin (see e.g. Theorem 8.6 in \cite{ether-book}) links \eqref{psi-eq} with $\theta =0$ to the support of super Brownian motion.} In contrast, in the sub-near-critical case (or the super-near-critical case and conditioned on extinction),  the tail distribution of $M/\sqrt{n}$ is similar to a Gumbel distribution.
\end{remark}

\begin{remark}\label{rmk:diff_asymptoticswith_LS}
In \cite{LS15}, the convergence \eqref{eq:conv_M}
was established by first proving the convergence of a  complicated object $\lim_{x\to\infty} w_\infty(x + y/w_\infty(x))/w_\infty(x)$, where $w_\infty(x)=P_1(M>x)$ for any $x$;
 see equation (23) therein for the precise statement.
In this paper, we prove the convergence of $P^{\theta/n}_n(M\geq \sqrt{n}x)$ directly.
\end{remark}

\paragraph{Organization of the paper:}
The rest of this paper is organized as follows. In Section \ref{section-F-K-w},  we establish a discrete Feynman-Kac formula for the tail distribution of the maximal displacement, which will be used in Section~\ref{sec-conv} to establish the tightness of $(nw_{n\cdot}^{\th/n}(\sqrt{n}\cdot))$, where $w_{k}^{\th/n}(x)=P_{1}^{\th/n} (M_{k}\geq x)$ for each~$k$ and $x$.  In Section \ref{section-pde-lim}, we identify the limit as a unique solution to a nonlinear parabolic PDE with infinite boundary condition, based on which we establish Theorem~\ref{theorem-weak-con} and Corollary \ref{corol-u-x}. In Section \ref{sec:travel_speed} we prove  Theorem~\ref{theorem-speed}, and in Section \ref{sec:Minfty} we prove Theorem \ref{theorem-u-inf}.

\section{A Discrete Feynman-Kac Formula } \label{section-F-K-w}
Recall that $F^{\th/n}_{B}$ stands for an offspring distribution with mean $1+\th/n$. Denote by $f^{\th/n}$ the probability generating function of $F^{\th/n}_{B}$.
Define
\bq \label{Q-\n}
Q^{\th/n}(s)=1-f^{\th/n}(1-s), \q \mbox{for all } s\in[0,1].
\eq
The function is  increasing and concave with $Q^{\th/n}(0)=0, (Q^{\th/n})'(0)=1+\th/n$ and $Q^{\th/n}(1) = 1-p_0^{\th/n} < 1$.
We also define
\be \label{w-def}
w_{k}^{\th/n}(y)=P_{1}^{\th/n} (M_{k}\geq y), \  y \in \zz{Z}, \ k \in\zz{Z}_{\geq 0}.
\ee
The derivation of the Feynman-Kac formula uses ideas from Section 2.2 in~\cite{LS15}.
The following lemma (see Lemma~4.1 in \cite{NZ17}) gives a convolution equation for~${w}_k^{\th/n}(\cdot)$ based on $Q^{\th/n}(\cdot)$ {and the random walk distribution $F_{RW} = {\{a_y\}_{y\in\zz{Z}}}$. It is obtained by conditioning on the first generation, in which a single particle first jumps according to the step distribution $ \{a_{k}\}$ and then reproduces according to $F_B^{{\theta/n}}$, which results in \hbox{i.i.d.} subtrees.

\begin{lemma} \label{2lamma-v-1}
For all $  k\geq 1$,
\[
 w^{\th/n}_{k}(x)=\sum_{y\in \zz{Z}}a_{y} Q^{\th/n}\big(   w^{\th/n}_{k-1}(x-y) \big).
 \]
\end{lemma}

Recall that $F^{\th/n}_{B}$ has a bounded third moment, hence by the Taylor expansion of~$Q^{\th/n}(\cdot)$ at $s=0$ we have
\bq \label{Q-\n1}
 Q^{\th/n}(s) = \Big(1+\frac{\th}{n}\Big)s -\frac{1}{2}{\wt\sigma^2}(\th/n) s^{2}+O(s^{3}),
\eq
where
\be \label{tilde-sigma}
\wt  \sigma^2(\th/n) =  \sigma^{2}(\th/n) +\Big(1+\frac{\th}{n}\Big)^{2}-\left(1+\frac{\th}{n}\right).
\ee
Define
\bq \label{2h-\n}
h^{\th/n}(s)=\Big(1+\frac{\th}{n}\Big)s- Q^{\th/n}(s) = \frac{1}{2}{\wt\sigma^2}(\th/n)  s^{2}+O(s^{3}),
\eq
and
\bq \label{2H-def1}
H^{\th/n}(s)=\frac{h^{\th/n}(s)}{(1+\th/n)s}=\frac{\wt\sigma^2(\th/n) }{2(1+\th/n)} s+O(s^{2}).
\eq
Lemma \ref{2lamma-v-1} can be rewritten as the following, which is more convenient for our purpose.
\begin{lemma}  \label{2lemma-v-id}
For all integers $  x\geq 1$ and $ k \geq 1$,
\[
 w^{\th/n}_{k}(x)=\Big(1+\frac{\th}{n}\Big)\sum_{y \in \zz{Z}}a_{y}  w^{\th/n}_{k-1}(x-y)-\sum_{y \in \zz{Z}}a_{y}h^{\th/n}\big( w^{\th/n}_{k-1}(x-y)\big).
\]
\end{lemma}

We will also need the following result on the boundedness and monotonicity of $H^{\th/n}$ (see Lemma 4.3 in \cite{NZ17}):
\be \label{H-bnd}
0\leq H^{\th/n}(s) \leq \frac{\th/n+p^{\th/n}_{0}}{1+\th/n} \leq 1, \quad \textrm{for all } s\in [0,1].
\ee
We denote by $\{ \W_{k}\}_{n \geq 0}$ a random walk on $\zz{Z}$ with the following law:
\be \label{reflected-rw}
P(\W_{k+1}-\W_{k}=y\ |\ \W_{k}, \W_{k-1},...)=a_{-y}, \ y\in \zz{Z};
\ee
in other words, $\{\W_{k}\}_{k \geq 0}$ is a reflection of $W$, the random walk associated with our branching system.
We use $P_{x}$ and $E_{x}$ to denote the probability measure and expectation of $\{\W_{k}\}_{k \geq 0}$ with $\W_{0}=~x$, and omit the {subscript} when $x=0$ (and when there is no confusion). {Moreover, to improve readability, we often abbreviate the notation $E_{[ \sqrt{n}x ]}$ to  $E_{ \sqrt{n}x}$, $\W_{[nt]}$ to $\W_{ nt }$, $w_{[nt]}^{\th/n_{}}$ to $w_{nt}^{\th/n_{}}$, etc.
}
We also denote by $\mathcal{F}^{\W}=(\mathcal{F}^{\W}_{k})_{k\geq 0}$ the natural filtration of $\{\W_{k}\}_{k \geq 0}$. \medskip \\

For any $ 0 \leq x \leq y<\infty $, define the stopping times
\be  \label{tau-y-ex}
\bar \tau_{y} =\min\{k\geq 0: \W_{k}\leq y\}, \q \tau_{y} =\min\{k\geq 0: \W_{k}\geq y\},
\ee
and
$$
\bar \tau_{x,y} = \bar \tau_{x} \wedge  \tau_{y}, \q \bar \tau:=\bar \tau_{0}.
$$

Further define for each $m\geq 0$ and $0\leq k \leq m$,
\be \label{2mrt-z}
\aligned
Y^{(n)}_{k}
=&\Big(1+\frac{\th}{n}\Big)^{k} w_{m-k}^{\th/n}(\W_{k})\mathds{1}_{\{\bar \tau\geq k\}} \prod_{j=1}^{k}\big[1-H^{\th/n}( w_{m-j}^{\th/n}\big(\W_{j})\big)\big] \\
& +\sum_{i=1}^{k-1}\Big(1+\frac{\th}{n}\Big)^{i-1}(1-p^{\th/n}_{0})\mathds{1}_{\{\bar \tau=i \}}\prod_{j=1}^{i-1}\big[1-H^{\th/n}( w^{\th/n}_{m-j}\big(\W_{j})\big)\big],
\endaligned
\ee
where we use the convention that
for any $k\leq 0,$ $\sum_{j=1}^{k}=0$ and $\prod_{j=1}^{k}=1$. In particular, ${Y^{(n)}_0}=w^{\th/n}_{m}(\W_0)$.

Similarly to Lemma 4.4 in \cite{NZ17}, we have that $Y^{}=\{Y^{(n)}_{k}\}_{0\leq k\leq {m}}$ is a martingale. To recall why this holds, define
\bd
\aligned
Y^{(n),1}_{k}
:=&\Big(1+\frac{\th}{n}\Big)^{k} w_{m-k}^{\th/n}(\W_{k})\mathds{1}_{\{\bar \tau\geq k\}} \prod_{j=1}^{k}\big[1-H^{\th/n}( w_{m-j}^{\th/n}\big(\W_{j})\big)\big] \\
Y^{(n),2}_{k} & =\sum_{i=1}^{k-1}\Big(1+\frac{\th}{n}\Big)^{i-1}(1-p^{\th/n}_{0})\mathds{1}_{\{\bar \tau=i \}}\prod_{j=1}^{i-1}\big[1-H^{\th/n}( w^{\th/n}_{m-j}\big(\W_{j})\big)\big],
\endaligned
\ed
Note that $Y^{(n),2}_{k+1} \in \mathcal F^{\W}_k$, so $Y$ is a martingale if and only if
\bd
E(Y^{(n),1}_{k+1}|\mathcal F^{\W}_k) = Y^{(n),1}_{k} +Y^{(n),2}_{k} -Y^{(n),2}_{k+1},
\ed
which is verified by using Lemma \ref{2lemma-v-id} and considering $\bar \tau<k$, $\bar \tau=k$ and $\bar \tau > k$, respectively.

\begin{lemma}  \label{2lemma-mart-z}
If $\W_0=x\geq 0$, then $Y^{}$ is a martingale with respect to $\mathcal F^{\W}$.
\end{lemma}

The following lemma gives  a discrete Feynman-Kac formula.
\begin{lemma}\label{lemma-disc-f-c-upto}
For any $m \in\zz{N}$, $0\leq k\leq m$ and  $ 0 \leq  y\leq x< z\leq \infty$, we have
\begin{itemize}
\item[(i)]
\bd
\begin{aligned}
w^{\th/n}_{m}(x)
=&E_{x}\Big(\Big(1+\frac{\th}{n}\Big)^{\bar\tau_{y,z}\wedge (m-k)} w_{m-(m-k)\wedge \bar \tau_{y,z}}^{\th/n}(\W_{\bar\tau_{y,z}\wedge (m-n)})  \\
& \qquad \times \prod_{j=1}^{\bar\tau_{y,z}\wedge (m-k)}\big[1-H^{\th/n}\big( w_{m-j}^{\th/n}(\W_{j})\big)\big] \Big).
\end{aligned}
\ed
\item[(ii)]
\bd
\begin{aligned}
w^{\th/n}_{m}(x)
=&E_{x}\Big(\Big(1+\frac{\th}{n}\Big)^{\bar\tau_{y,z}\wedge m} w_{m- \bar \tau_{y,z}\wedge m}^{\th/n}(\W_{\bar\tau_{y,z}\wedge m})  \\
& \qquad \times \prod_{j=1}^{\bar\tau_{y,z}\wedge m}\big[1-H^{\th/n}\big( w_{m-j}^{\th/n}(\W_{j})\big)\big] \Big).
\end{aligned}
\ed
\end{itemize}
\end{lemma}
\begin{proof}
Note that $\bar \tau_{y,z} \leq \bar\tau$ for every $ 0 \leq  y< z\leq \infty $. The conclusions follow by taking the stopping times $\bar\tau_{y,z}\wedge (m-k)$ and $\bar\tau_{y,z}\wedge m$ and applying the optional stopping theorem to the martingale $\{Y^{(n)}_{k}\}$.

\end{proof}

\section{Tightness } \label{sec-conv}
In this section, we establish the tightness of the function sequence $\left(n w^{\th/n}_{[ n_{} t]}(\sqrt{n} x)\right)_{n\geq 1}$. Because  at $x=0$, $n w^{\th/n}_{[ n_{} t]}(\sqrt{n} x)=n\to\infty$, the sequence of functions  cannot be tight when allowing $(t,x)$ to vary over the whole domain $\{(t,x): t\geq0,x\geq 0\}$. Special treatments are needed to deal with such a singularity.

We start with some exponential bounds on the distribution of the maximum of $W$.

\begin{lemma}  \label{Lemma-uni2} {Let $Y_{1},Y_{2},...$ be \mbox{i.i.d.} random variables with mean $0$ and $P(Y_{1}\geq R) =0$ for some $R$. Let $S_{n}=Y_{1}+...+Y_{n}$.}
 \begin{itemize}
\item[(i)] There exist constants $C_{\ref{Lemma-uni2}},\beta_{\ref{Lemma-uni2}}>0$ such that
\bn  \label{mr21}
 P\Big(\max_{i=0,...,n} |S_{i}|\geq s \sqrt{n} \Big) \leq C_{\ref{Lemma-uni2}}e^{-\beta_{\ref{Lemma-uni2}}s^{2}}, \q \mbox{for all } n \geq 0, \ s>0.
\en
\item[(ii)] There exist constants $C_{\ref{Lemma-uni2}}',\beta'_{\ref{Lemma-uni2}}>0$ such that
 \bn \label{mr24}
 P\Big( \max_{i=0,...,n} S_{i} \leq s\sqrt{n}\Big) \leq C'_{\ref{Lemma-uni2}}e^{-\beta'_{\ref{Lemma-uni2}}/s^{2}}, \q \mbox{for all } n\geq 0, \ s>0.
 \en
\end{itemize}
\end{lemma}
\begin{proof}
(i) is a special case of as Corollary A.2.7 in \cite{Lawler2010}.

(ii) The result follows from equation (2.51) in Proposition 2.4.5 of \cite{Lawler2010}; {see also Exercise 2.7 therein.}  \end{proof}

In the following lemma, we compute the probability that $X^{}$ (under $P_{n}^{\th/n}$) dies out as $n\rr \infty$.
\begin{lemma}  \label{lemma-ext}
Assume that $F^{\theta/n}_{B}$ satisfies (\ref{assum-sig-gam}). When $\th>0$, we have
\bn
\lim_{n\rr\infty}P_{n}^{\th/n}\big(X^{} \textrm{ dies out }\big) = e^{-2\th/\sigma^2}.
\en
\end{lemma}
\begin{proof}
Recall that $f^{\th/ n}$ is the probability generating function of $F^{\th/ n}_{B}$.
Let $q_{n}$ be the smallest non-negative root of the equation $f^{\theta/n}(q)=q$.
By Theorem 2 in Chapter I.A.5 of \cite{Athreya-1972},
$
P^{\theta/ n}_{1}\big(X \textrm{ dies out }\big) = q_n.
$
By {(\ref{assum-sig-gam})}, for  $q\in (0,1)$,
\begin{equation} \label{eqn:taylor_pgf}
f^{\th/n}(q) = 1+ (q-1)\Big(1+\frac{\th}{n}\Big)+ \frac{\sigma^2}{2}((1-q)^{2}) + o\big((1-q)^{2}\big).
\end{equation}
It follows  that
\bq  \label{b222}
P^{\theta/ n}_{1}\big(X \textrm{ dies out }\big) = q_{n}=1-\frac{2\th}{n\sigma^2} +o\big(1/n\big),
\eq
and
\bq \label{b3}
\lim_{n\rr\infty }P^{\th/n}_{n}\big(X^{} \textrm{ dies out }\big) = \lim_{n\rr\infty} \left(1 -\frac{2\th}{n\sigma^2}+o(n^{-1})\right)^{n}
= e^{-2\th/\sigma^2}.\nonumber
\eq

\end{proof}

Before stating our next lemma, we recall the duality principle which states that a supercritical branching process conditional on extinction has the same distribution as its dual subcritical process; see, for example, Theorem 3 in Chapter {I.D.12} in \cite{Athreya-1972}. Specifically,
let $\overline Z= \{\overline Z_{n}\}_{n\geq 0}$ be a Galton-Watson process with $\ol{Z}_0=1$ and an offspring distribution $\overline F_{B} = \{\ol{p}_i\}_{i\geq 0}$ that has mean $\mu>1$ and $\ol{p}_0>0$.
Define
\bn
B^{}=\{\omega:\overline Z^{}_{n}(\omega)=0 \textrm{ for some } n\geq 1\}
\en
to be the event of extinction, and let $q=P(B)\in (0,1)$.
Then the duality principle says that the process $\{\overline Z_{n}\}_{n\geq 0}$
conditional on  event $B$ has the same distribution as a subcritical Galton-Watson branching process $\{ \wt Z_{n}\}_{n\geq 0}$  with $\wt Z^{}_{0}=1$
and
\bn
E^{}\big(\xi^{\wt Z^{}_{1}}\big)  \equiv  \frac {\ol{f}(\xi q)}{q^{}}, \q \xi \in (0,1),
\en
where $\ol f^{}$ denotes the probability generating function of $\overline F_{B}$.

In the following lemma, we derive an exponential bound on $\sup_{n\geq 1}n_{}w^{\th/n_{}}_{nt}(\sqrt{n} x)$.
\begin{lemma} \label{tight-n-w}
\begin{itemize}
\item[(i)] For any $\delta>0,$
$$
\sup_{x \geq \dl} \sup_{n\geq 1}nw_{nt}^{\th/n_{}}(\sqrt{n}x) <  \infty;
$$
\item[(ii)]
there exists $\beta>0$ such that
$$
\lim_{x \rr \infty} \sup_{n\geq 1} e^{\beta x^2/t}\cdot nw_{nt}^{\th/n_{}}(\sqrt{n}x) =0, \quad \textrm{for all } t>0.
$$
\end{itemize}
\end{lemma}
\begin{proof}
(i) We shall only prove the result when $\th \geq 0$.
Recall that $M= \sup_{n\geq 0} M_{n}$ stands for the maximal displacement over all generations.
From Theorem 1 in \cite{LS15}, which applies to critical branching random walks, we have
\be \label{LS-res}
\sup_{x \geq 1}x^2P^{0}_1(M \geq x)  <\infty.
\ee
Therefore, if we denote by $B$ the event of extinction of the branching random walk $X$, then by the duality principle we get
$$
\sup_{x \geq \dl} \sup_{n\geq 1} nx^2P^{\theta/ {n}}_1(M \geq \sqrt{n}x \,| \, B)  <\infty.
$$
Moreover, by (\ref{b222}), there exist positive constants $c_1$ and $c_2$ independent of $n$ and $x$ such that
\[
\aligned
nw^{\theta/n}_{nt}(\sqrt{n} x)
\leq & {n P_1^{\theta/n}(B^c) + nP_1^{\theta/n}(M>\sqrt{n}x\,  |\, B) }\\
\leq & n\left(\frac{c_1}{n} + \frac{c_2}{nx^2}\right), \quad \textrm{for all } x\geq {\delta},\ n\geq 0,
\endaligned
\]
and we get (i).

(ii)
By Lemma \ref{lemma-disc-f-c-upto}(ii) and (\ref{H-bnd}), we have
\begin{equation} \label{ttr1}
\aligned
&nw^{\th/n}_{nt}(\sqrt{n}x)\\
\leq  &nE_{\sqrt{n}x}\Big(\Big(1+\frac{\th}{n}\Big)^{\bar\tau_{\sqrt{n}x/2}\wedge (nt)} w_{nt- \bar \tau_{\sqrt{n}x/2}\wedge (nt)}^{\th/n}(\W_{\bar\tau_{\sqrt{n}x/2}\wedge (nt)}) \Big) \\
\leq &  n e^{\theta^+ t} E_{\sqrt{n}x}\Big(w_{0}^{\th/n}(\W_{nt}) \mathds{1}_{\{\bar\tau_{\sqrt{n}x/2}\geq nt\}} \Big) \\
& \quad + n e^{\theta^+ t} E_{\sqrt{n}x}\Big(w_{nt}^{\th/n}(\W_{\bar\tau_{\sqrt{n}x/2}}) \mathds{1}_{\{\bar\tau_{\sqrt{n}x/2}< nt\}} \Big).
\endaligned
\end{equation}
Because $w^{\theta/n}_0(y) = 0$ for  $y\geq 1$, we have for all $x >0$ and $n\geq 1$,
\bd
n e^{\theta^+ t} E_{\sqrt{n}x}\Big(w_{0}^{\th/n}(\W_{nt}) \mathds{1}_{\{\bar\tau_{\sqrt{n}x/2}\geq nt\}} \Big) =0.
\ed
Recall that $\W$ has a range $R$. From the monotonicity of $w^{\theta/n}_{nt}(x)$ in $x$ and part~(i), we get that there exists a positive constant $C$  such that
\begin{equation}\label{eq:nw_bd_exit_prob}
\aligned
 & \q n e^{\theta^+ t}\, E_{\sqrt{n}x}\Big(w_{nt}^{\th/n}(\W_{\bar\tau_{\sqrt{n}x/2}}) \mathds{1}_{\{\bar\tau_{\sqrt{n}x/2}< nt\}} \Big) \\
 \leq &  e^{\theta^+ t} E_{\sqrt{n}x}\Big(nw_{nt}^{\th/n}(\sqrt{n}x/2-R) \mathds{1}_{\{\bar\tau_{\sqrt{n}x/2}< nt\}} \Big) \\
\leq & {C} e^{\theta^+ t} P_{\sqrt{n}x}( \bar\tau_{\sqrt{n}x/2}< nt ) \\
\leq & C C_{\ref{Lemma-uni2}}e^{\theta^+ t} e^{- \beta_{3,1}\frac{x^2}{t}},  \quad \textrm{for all } n\geq 1, \ x \geq 2R,
 \endaligned
\end{equation}
where we used Lemma \ref{Lemma-uni2}(i) in the last inequality.
The conclusion follows.
\end{proof}

The following lemmas are key ingredients in proving the tightness of of~$(n w^{\th/n}_{nt}(\sqrt{n} x))$.

\begin{lemma}  \label{lemma-pgf}
For any $T>0$, there exists $C(T)>0$ such that for all $x, y \in \zz{Z}/\sqrt{n}$ with $ |x-y| \leq 1$ and  $ t\in [0,T],$
\bd
\sup_{n\geq1} \Big|E_{\sqrt{n}x}\Big(\Big(1+\frac{\th}{n}\Big)^{{\tau_{\sqrt{n}y}}\wedge (nt)} \Big) -1\Big| \leq C(T)|x-y|.
\ed
\end{lemma}
\begin{proof}
We only need to prove for the case when $y>x$ because otherwise the LHS equals 0.
We will also only prove for the case when $\th\geq 0$; the case when~$\th<0$ can be proved similarly.

Note that
\[
\aligned
E_{\sqrt{n}x}\Big(\Big(1+\frac{\th}{n}\Big)^{\tau_{\sqrt{n}y}\wedge (nt)} \Big)
&=E_{\sqrt{n}x}\Big(\Big(1+\frac{\th}{n}\Big)^{{\tau_{\sqrt{n}y}\wedge (nt)}} \cdot \indic_{{\{\tau_{\sqrt{n}y}\leq n(y-x)} \} } \Big) \\
&\q +E_{\sqrt{n}x}\Big(\Big(1+\frac{\th}{n}\Big)^{\tau_{\sqrt{n}y}\wedge (nt)} \cdot \indic_{{\{\tau_{\sqrt{n}y}> n(y-x)}\}} \Big)\\
&=:I_{1}(n,y-x,t) +I_{2}(n,y-x,t).
\endaligned
\]
For $I_{1}(n,y-x,t)$,  we have
\bn
I_{1}(n,y-x,t)
\leq  e^{\th(y-x)}.
\en
Because $e^{x}$ is a Lipschitz function, we get
$$
I_{1}(n,y-x)-1 \leq C(y-x), \quad \textrm{for all }  0\leq y-x \leq 1.
$$
About  term $I_{2}(n,y-x)$, using Lemma \ref{Lemma-uni2}(ii) we get
\bn
I_{2}(n,y-x)& \leq& e^{\th t} P_{\sqrt{n}x}\big(\tau_{\sqrt{n}y} > n(y-x)) \\
&\leq & C(T)e^{-\beta'/(y-x)},
\en
and the result follows.
\end{proof}
\begin{lemma}  \label{lemma-spat-reg}
For any $0<t_0<T<\infty$ and $x_{0}>0$, there exist $\n_0>0$ and $C=C(t_0,T,x_0)>0$ such that for all $n \geq \n_0$ and $t\in [t_0,T]$ we have
\bn
n\big|w^{\th/n}_{nt}(\sqrt{n}y)-w^{\th/n}_{nt}(\sqrt{n}x) \big| &\leq& C|y-x|, \\
&& \textrm{for all }  \ x_{0}< x, y \in \zz{Z}/\sqrt{n}  \mbox{ with }|x-y|\leq 1.
\en
\end{lemma}
\begin{proof}
Let $x_{0}>0$, $x_{0}< x< y <\infty$ and $t\geq t_0$.
It follows from Lemma \ref{lemma-disc-f-c-upto}(ii) with $m=nt$ and \eqref{H-bnd} that
\be \label{rf1}
\begin{aligned}
&\qq w^{\th/n}_{nt}(\sqrt{n}x)  \\
&\leq E_{\sqrt{n}x}\Big(\Big(1+\frac{\th}{n}\Big)^{\bar\tau_{\sqrt{n}x/2,\sqrt{n}y}\wedge (nt)} w_{nt- \bar \tau_{\sqrt{n}x/2,\sqrt{n}y}\wedge (nt)}^{\th/n}(\W_{\bar\tau_{\sqrt{n}x/2,\sqrt{n} y}\wedge (nt)}) \Big) \\
&\leq E_{\sqrt{n}x}\Big(\Big(1+\frac{\theta^+}{n}\Big)^{{\tau_{\sqrt{n}y} \wedge (nt) } } w_{nt}^{\th/n}(\sqrt{n}y) \\
&\quad + E_{\sqrt{n}x}\Big(\Big(1+\frac{\theta^+}{n}\Big)^{ nt} w_{nt}^{\th/n}(\W_{ \bar \tau_{\sqrt{n}x/2}\wedge  (nt)}) \mathds{1}_{\{\tau_{\sqrt{n} y} > (nt)\wedge \bar\tau_{\sqrt{n} x/2} \}}\Big) \\
&=:I_{1}(n,x,y,t)+I_{2}(n,x,y,t),
\end{aligned}
\ee
where {in the second inequality} we  used the monotonicity of $w^{\th/n}_{nt}(x)$ in $x$.

We first handle $I_{1}(n,x,y,t)$. By Lemma \ref{lemma-pgf}, for all $0\leq y-x\leq 1$,
\be \label{rf2}
\big | I_{1}(n,x,y,t)-w_{nt}^{\th/n}(\sqrt{n}y) \big |
\leq C(T)(y-x)w_{nt}^{\th/n}(\sqrt{n}y),
\ee
which, by Lemma \ref{tight-n-w}(i), is bounded by $C(y-x)/n$.

Next we bound $I_{2}(n,x,y,t)$. Using the monotonicity of $w^{\th/n}_{nt}(x)$ in $x$ again we have
\begin{equation}\label{eq:I2}
\aligned
&I_{2}(n,x,y,t)\\
 \leq & e^{\theta^+ t }E_{\sqrt{n}x}\Big(w_{nt}^{\th/n}(\W_{ \bar \tau_{\sqrt{n}x/2}\wedge  (nt)})\mathds{1}_{\{\tau_{\sqrt{n} y} > (nt)\wedge \bar\tau_{\sqrt{n} x/2})\}}\Big) \\
\leq & e^{\theta^+ t}w_{nt}^{\th/n}(\sqrt{n}x/2-R)  P_{\sqrt{n}x}\big({\tau_{\sqrt{n} y} }> (nt)\wedge \bar\tau_{\sqrt{n} x/2}\big) \\
\leq & e^{\theta^+ t}w_{nt}^{\th/n}(\sqrt{n}x/2-R) \Big(P_{\sqrt{n}x}\big({\tau_{\sqrt{n} y}} > nt\big)  + P_{\sqrt{n}x}\big(\tau_{\sqrt{n} y} >  \bar\tau_{\sqrt{n} x/2}\big)\Big) \\
\leq & e^{\theta^+ t}w_{nt}^{\th/n}(\sqrt{n}x/2-R)\Big(C_{1}e^{-C_{2}t/(x-y)^{2}}+\frac{\sqrt{n}(y-x) + R}{\sqrt{n}(y-x/2)}\Big),
\endaligned
\end{equation}
where in the last inequality we used Lemma \ref{Lemma-uni2}(ii) and the Optional Stopping Theorem.
The conclusion follows from Lemma \ref{tight-n-w}(i).
\end{proof}

\begin{remark}\label{rmk:I_2_5th_moment}
The bound in \eqref{eq:I2} is a place where we need the random walk to have a finite 5th moment. Specifically, when $F_{RW}$ has an unbounded support but a finite 5th moment,  we can estimate term $I_{2}(n,x,y,t)$ as follows. Note that to prove the conclusion in the lemma, by the triangular inequality, it suffices to prove for the case when $y=x+1/\sqrt{n},$ and we want to show that $n I_{2}(n,x,y,t)=O(1/\sqrt{n})$.  We have
\begin{equation}\label{eq:I2_2}
\aligned
&I_{2}(n,x,y,t) \\
\leq & e^{\theta^+ t }E_{\sqrt{n}x}\Big(w_{nt}^{\th/n}(\W_{ \bar \tau_{\sqrt{n}x/2}\wedge  (nt)})\mathds{1}_{\{\tau_{\sqrt{n} y} > (nt)\wedge \bar\tau_{\sqrt{n} x/2})\}}\Big) \\
\leq &e^{\theta^+ t}w_{nt}^{\th/n}(\sqrt{n}x/4) \Big(P\big(\tau_1 > nt\big)  + P\big(\tau_1 >  \bar\tau_{-\sqrt{n} x/2}\big)\Big) \\
&\quad+  e^{\theta^+ t}
{P_{\sqrt{n}x}\big(\W_{\bar\tau_{\sqrt{n} x/2}}< \sqrt{n}x/4\big)} \\
\leq & e^{\theta^+ t}w_{nt}^{\th/n}(\sqrt{n}x/4)\Big(C_{1}e^{-C_{2} n t}+C_3/\sqrt{n}\Big) + {e^{\theta^+ t}}C_4/(n\sqrt{n}),
\endaligned
\end{equation}
where in the last inequality we used the following estimates.
\begin{enumerate}[(i)]
  \item $P\big(\tau_{1} >  \bar\tau_{-k}\big) = O(1/k)$ as $k\to\infty$. To see this, note that by the Optional Stopping Theorem,
  \[
    E(\W_{\bar\tau_{1,-k}} \mathds{1}_{\{\tau_{1} <  \bar\tau_{-k}\}} ) = - E(\W_{\bar\tau_{1,-k}} \mathds{1}_{\{\tau_{1} >  \bar\tau_{-k}\}} )
    \geq k P\big(\tau_{1} >  \bar\tau_{-k}\big).
  \]
  Moreover, by exercise 6, p.232 of \cite{Spitzer}, $E(\W_{\tau_1})<\infty$ if $F_{RW}$ has finite variance, hence $P\big(\tau_{1} >  \bar\tau_{-k}\big) = O(1/k)$.
  \item $P\big(\W_{\bar \tau_{-k}} \leq - \alpha k\big) = O(1/k^{3})$ as $k\to\infty$. This follows from Lemma 10 (and its proof) in \cite{LS15}, by which we have, if  $F_{RW}$ has  a finite 5th moment, then the limiting overshoot distribution has a finite 3rd moment. \\
\end{enumerate}
\end{remark}

\begin{lemma}  \label{lemma-reg-time}
For any $0<t_0<T<\infty$ and $x_{0}>0$, there exist $\dl>0$, $\n_0>0$ and $C(t_0,T,x_0)>0$ such that for all $n>\n_0$, $t\in [t_0,T]$, $nt \leq m \leq n (t+\dl)$ and $ x_{0}< x<\infty$, we have
\bn
n\big|w^{\th/n}_{m}(\sqrt{n}x)-w^{\th/n}_{nt}(\sqrt{n}x) \big| \leq \eps.
\en
\end{lemma}
\begin{proof} From monotonicity it is enough to prove the lemma when $m =[ n(t+\dl)] $.
Let $ x_{0}< x<\infty$ and let $\xi>0$ be a small number to be chosen later.  From the bound \eqref{H-bnd} and  Lemma \ref{lemma-disc-f-c-upto}(i) with $y=x-\xi$, $z=\infty$ and $k=[nt]$ we get
\be \label{rr0}
\begin{aligned}
& w^{\th/n}_{m}(\sqrt{n}x) \\
\leq& E_{\sqrt{n}x}\Big(\Big(1+\frac{\th}{n}\Big)^{\bar\tau_{\sqrt{n}(x-\xi)}\wedge (m-[nt] )} w_{m-(m-[nt])\wedge \bar\tau_{\sqrt{n}(x-\xi)}}^{\th/n}(\W_{{\bar\tau}_{\sqrt{n}(x-\xi)}\wedge (m-[nt])})\Big) \\
\leq& E_{\sqrt{n}x}\Big(\Big(1+\frac{\th}{n}\Big)^{m-[nt]}w_{[nt]}^{\th/n}(\W_{m-[nt]})\mathds{1}_{\{\bar \tau_{\sqrt{n}(x-\xi)}> m-[nt]\}}\Big) \\
&\quad + E_{\sqrt{n}x}\Big(\Big(1+\frac{\theta^+}{n}\Big)^{m-[nt] } w_{m- \bar\tau_{\sqrt{n}(x-\xi)}}^{\th/n}(\W_{\bar\tau_{\sqrt{n}(x-\xi)}})\mathds{1}_{\{\bar\tau_{\sqrt{n}(x-\xi)} \leq m-[nt]\}}\Big) \\
=&: J_{1}(m,n,x)+J_{2}(m,n,x).
\end{aligned}
\ee
Note that $ m-[nt] \leq \dl n +1$. Using the monotonicity of $w^{\th/n}_{k}(x)$ in  $x$ we get that
\bd
\begin{aligned}
 J_{1}(m,n,x) &\leq e^{\theta^+(\dl+1/n) }E_{\sqrt{n}x}\Big(w_{nt}^{\th/n}(\W_{m-[nt]})\mathds{1}_{\{\bar\tau_{\sqrt{n}(x-\xi)}> m-[nt]\}}\Big)  \\
 &\leq e^{\theta^+(\dl+1/n) }w_{nt}^{\th/n}(\sqrt{n}(x-\xi)) .
\end{aligned}
\ed
By Lemma \ref{lemma-spat-reg} and Lemma \ref{tight-n-w}(i), if $\xi$ and $\dl$ are small enough then
\be \label{rr1}
 J_{1}(m,n,x) \leq e^{\theta^+(\dl+1/n) }w_{nt}^{\th/n}(\sqrt{n}x) +\frac{\eps}{n}
  \leq w_{nt}^{\th/n}(\sqrt{n}x) +\frac{2\eps}{n}.
\ee
As to $J_{2}(m,n,x)$, using the monotonicity of $w^{\th/n}_{k}(x)$ in $k$ and $x$, the finite range of $\W$ and noting that $m-[nt] \leq \dl n+1$, we have
\be \label{j-2-fr}
\begin{aligned}
 J_{2}(m,n,x) &\leq e^{\theta^+(\dl+1/n) }E_{\sqrt{n}x}\Big(w_{m- \bar\tau_{\sqrt{n}(x-\xi)}}^{\th/n}(\W_{\bar \tau_{\sqrt{n}(x-\xi)}})\mathds{1}_{\{\bar \tau_{\sqrt{n}(x-\xi)} \leq \dl n+1\}}\Big) \\
 &\leq e^{\theta^+(\dl+1/n) }E_{\sqrt{n}x}\Big(w_{m}^{\th/n}(\W_{\bar \tau_{\sqrt{n}(x-\xi)}})\mathds{1}_{\{\bar \tau_{\sqrt{n}(x-\xi)} \leq \dl n+1\}}\Big) \\
  &\leq e^{\theta^+(\dl+1/n) }w_{n(t+\dl)}^{\th/n}\big(\sqrt{n}(x-\xi)-R\big)\cdot P_{\sqrt{n}x}\big(\bar \tau_{\sqrt{n}(x-\xi)} \leq \dl n+1\big).
\end{aligned}
\ee
By Lemma \ref{Lemma-uni2}(i), there exist $C_2>0$ and $
\beta>0$ such that for all $x\geq x_0$,
\bn
P_{\sqrt{n}x}\big(\bar\tau_{\sqrt{n}(x-\xi)} \leq  n \dl + 1\big)
&\leq & C_2e^{-\beta \frac{\xi^2}{\dl} }.
\en
By choosing $\dl$ to be small enough and using  Lemma \ref{tight-n-w}(i) we get
\be \label{rr2}
J_{2}(m,n,x) \leq  \frac{\eps}{n}.
\ee
The conclusion follows.
\end{proof}

\begin{corollary}  \label{coroll-unif}
For any $\eps>0$, $0<t_0<T<\infty $ and $x_{0}>0$, there exist $\dl >0$ and $\n_0>0$ such that for all $n>\n_0$, $t_0\leq s, t\leq T$ with $|s-t|\leq \dl$ and $ x, y \in [x_0,\infty)$ satisfying $|x-y|\leq \dl $, we have
\[
n\big|w^{\th/n}_{nt}(\sqrt{n}y)-w^{\th/n}_{ns}(\sqrt{n}x) \big| \leq \eps.
\]
\end{corollary}

\begin{proposition}  \label{prop-conv-sub}
For any sequence  $\{n_{i}\}$ of positive integers that increase to infinity, there exists a subsequence $\{n_{i_j}\}$ along which the functions $\left( n_{i_j}w^{\th/n_{i_j}}_{[n_{i_j} t]}\left(\sqrt{n_{i_j}} x \right)\right)_{t>0, x>0}$ converge. The convergence is uniform over any compact region inside $\{(t,x): t >0, x>0\}$. Moreover,
any subsequential limiting function $\phi(t,x)$ is increasing in~$t$, decreasing and Lipschitz in $x$ with
\be \label{bnd-cond1}
\lim_{x\rr \infty}\phi(t,x)=0,  \quad \textrm{for each } t >0.
\ee
\end{proposition}
\begin{proof}
The conclusions follow from Lemma \ref{tight-n-w}, Lemma \ref{lemma-spat-reg}, Corollary \ref{coroll-unif}, Arzel\`{a}-Ascoli Theorem and a standard diagonal argument.
\end{proof}

\section{Scaling Limit}\label{section-pde-lim}
In this section, we prove that the limiting function $\phi(\cdot)$ from Proposition \ref{prop-conv-sub} satisfies a nonlinear  PDE with an infinite boundary condition at $x=0$. Then we prove uniqueness of solutions to the PDE.

The main difficulty with the analysis of the limiting function is that it satisfies a nonlinear parabolic PDE with singularity at the origin; see Corollary \ref{corollary-pde} and Lemma \ref{lemma-boundary}. In Proposition \ref{prop-phi-con} and Lemma \ref{lemma-boundary}, we develop probabilistic arguments  to characterize the limiting function. We further use PDE methods to prove uniqueness of the solution to equation (\ref{kpp-uniq}) (Proposition \ref{prop-uniq}) and to link the solution with the tail distribution of the support of super-Brownian motion (Theorem \ref{theorem-weak-con}).

For any continuous process  $\{Y_{t}\}_{t\geq 0}$, define
\be  \label{tau-B}
\bar \tau^{Y}_{x}  =\min\{t\geq 0: Y_{t} \leq x\}, \mbox{ and }\bar \tau^{Y} =\bar \tau^{Y}_{0}.
\ee
The following proposition gives a Feynman-Kac representation of $\phi(t,x)$.

\begin{proposition} \label{prop-phi-con}
Let $(\phi(t,x))_{t>0, x>0}$ be any sub-sequential limiting function from Proposition \ref{prop-conv-sub}. Then for any $x_0>0$ and for all  $x > x_0, \, t>0,$ we have
\begin{equation}\label{eq:phi_FK}
\aligned
&\phi(t,x) =\\
 E_{x} & \left(\exp\left( \int_{0}^{t\wedge \bar \tau^{\sigma_R B}_{x_0}}\left(\th-\frac{\sigma^{2}}{2}\phi\left(t-s,\sigma_R B_{s}\right)\right)ds\right)
  \phi\left(t - t\wedge \bar \tau^{\sigma_R B}_{ x_0} , \sigma_R B_{t\wedge \bar \tau^{\sigma_R B}_{ x_0}}\right)\right),
\endaligned
\end{equation}
where under $P_{x}$, $\{B_{t}\}_{t\geq 0}$ is a standard Brownian motion  starting at $x/\sigma_R$.
\end{proposition}
\begin{proof}
For any $x>x_0$ and $t>0$, by Lemma \ref{lemma-disc-f-c-upto}(ii), we have
\[
\aligned
w^{\th/n}_{ nt }(\sqrt{n}x) =&  E_{\sqrt{n} x}\Big(\Big(1+\frac{\th}{n}\Big)^{nt\wedge \bar \tau_{\sqrt{n} x_0}}w_{n t-nt\wedge \bar\tau_{\sqrt{n} x_0} }^{\th/n}(\W_{nt\wedge  \bar\tau_{\sqrt{n} x_0}}) \\
 \quad &\times \prod_{j=1}^{nt\wedge\bar\tau_{\sqrt{n} x_0} }\big[1-H^{\th/n}( w_{n t-j}^{\th/n}\big(\W_{j})\big)\big] \Big).
\endaligned
\]
Using (\ref{2H-def1}) and that $\log(1-x) = -x+ O(x^{2})$ for $|x|<1$, we get
\be
\begin{aligned} \label{gf1}
n w^{\th/n}_{ n t }(\sqrt{n}x)
=&  E_{\sqrt{n} x}\Bigg(\left(1+\frac{\th}{n}\right)^{(nt)\wedge \bar \tau_{\sqrt{n}x_0}}\cdot nw_{n t-nt\wedge\bar  \tau_{\sqrt{n}x_0} }^{\th/n}(\W_{nt\wedge  \bar\tau_{\sqrt{n} x_0}})  \\
&\quad \times\exp\Bigg\{-\sum_{j=1}^{nt\wedge \bar \tau_{\sqrt{n} x_0}}\frac{\wt\sigma^2(\th/n) }{2(1+\th/n)} w_{n t-j}^{\th/n} (\W_{j})\\
&\qquad+\big(nt\wedge \bar \tau_{\sqrt{n} x_0}\big)O\Big(\big(w_{n t}^{\th/n}(\sqrt{n} x_0- R)\big)^{2}\Big)\Bigg\} \Bigg).
\end{aligned}
\ee
By Lemma \ref{tight-n-w}(i), the error term
\[
\left(nt\wedge \bar \tau_{\sqrt{n} x_0}\right)O\left(\left(w_{n t}^{\th/n}(\sqrt{n} x_0- R)\right)^{2}\right) = o(1).
\]
Moreover, by Donsker's invariance principle {(see, e.g., Theorem 8.2 in \cite{billingsley})}, under  $P_{\sqrt{n} x}$,  $\left({\W_{nt}}/(\sqrt{n}\sigma_R)\right)_{t\geq 0}$  converges to a standard Brownian motion $(B_{t})_{t\geq 0}$ starting at $x/\sigma_R$ as $n \rr \infty$. In addition,  $n^{-1}\bar \tau_{\sqrt{n}x_0}$  converges weakly to $\bar \tau^{\sigma_R B}_{ x_0}$; see, for example, Theorem 1 in \cite{Kennedy} and Theorem 3 in \cite{Mirak78}. Using the Skorokhod embedding theorem, strong Markov property and the fact that a Brownian almost surely takes both positive and negative values in any interval $[0,\delta]$, one can further strengthen the previous two convergences to be joint convergence.
Therefore, by our assumption on the convergence of $nw_{[nt] }^{\th/n}(\sqrt{n}x)$ to $\phi(t,x)$, we get
\be \label{kq1}
nw_{n t-nt\wedge \bar\tau_{\sqrt{n} x_0} }^{\th/n}\big(\W_{nt\wedge  \bar\tau_{\sqrt{n} x_0}}\big)
 \Rightarrow  \phi(t - t\wedge \bar \tau^{\sigma_R  B}_{ x_0} ,\sigma_{R}  B_{t\wedge \bar \tau^{\sigma_R B}_{ x_0}}), \quad \textrm{as } n\rr \infty.
\ee
Furthermore, from (\ref{assum-sig-gam}) and (\ref{tilde-sigma}) we get $ \wt\sigma^2(\th/n) \rr  \sigma^2$. By the same reasoning as for (\ref{kq1}) we get
\bq \label{lg1}
&&\frac{\wt\sigma^2(\th/n) }{2(1+\th/n)} \sum_{j=1}^{(nt)\wedge \bar \tau_{\sqrt{n}x_0}}w_{n t-j}^{\th/n}\big(\W_{j}) \nonumber \\
=&&\frac{\wt\sigma^2(\th/n) }{2(1+\th/n)} \cdot \frac{1}{n}\sum_{j=1}^{n (t\wedge  (n^{-1}\bar \tau_{\sqrt{n}x_0}))}  nw^{\th/n}_{n t-j}\Big(\sqrt{n}\frac{ \W_{j}}{\sqrt{n}}\Big)  \nonumber \\
\Rightarrow & &\frac{ \sigma^{2}}{2} \int_{0}^{t\wedge \bar \tau^{\sigma_R B}_{x_0}}\phi(t-s,\sigma_{R} {B_{s}})\,ds, \quad  \textrm{as } n\rr\infty. \nonumber
\eq
Finally, we have
$$
\Big(1+\frac{\th}{n}\Big)^{nt\wedge  \bar \tau_{\sqrt{n} x_0}}  \Rightarrow  \exp\left(\th\, t\wedge \bar \tau^{\sigma_R B}_{x_0}\right),  \quad  \textrm{as } n\rr\infty.
$$
Plugging the above limits into (\ref{gf1}), together with bounded convergence theorem we get the conclusion.
\end{proof}

\begin{corollary} \label{corollary-pde}
Suppose that $(\phi(t,x))_{t>0, x>0}$ is a sub-sequential limiting function from Proposition \ref{prop-conv-sub}. Then it  satisfies the following PDE:
\be \label{kpp}
\frac{\partial \phi}{\partial t}(t,x) = \frac{\sigma_{R}^{2}}{2} \frac{\partial^{2}\phi}{\partial x^{2} }(t,x)+\th\phi(t,x) -\frac{ \sigma^{2}}{2}\phi^{2}(t,x), \quad t >0, \ x >0.
\ee
\end{corollary}

\begin{proof}
This follows from Kac's theorem; see, for example, Section 2.6 in \cite{ItoMckean} or Theorem 4.1 in Section 3.4 of \cite{Pinsky95book}. 
\end{proof}

In the following lemma, we derive the initial and boundary conditions of~$\phi$ from Proposition \ref{prop-conv-sub}.
\begin{lemma}  \label{lemma-boundary}
\begin{itemize}
\item [(i)] For any $x>0$,
\bn
\lim_{t\dr 0} \phi(t,x) = 0.
\en
\item [(ii)] For any $t>0$,
\bn
\lim_{x\dr 0} \phi(t,x) = \infty.
\en
\item [(iii)]  For any $T>0$,
\bn
\lim_{x\rr \infty}   \phi(t,x) = 0, \quad \textrm{uniformly on } [0,T].
\en
\end{itemize}
\end{lemma}
\begin{proof}
(i) It is sufficient to show that for any $x>0$ and $\eps>0$, there exists $t_0>0$ such that
 \bq \label{tt1}
 \sup_{n\geq 1}nw_{nt}^{\th/n}(\sqrt{n}x) \leq \eps, \quad \textrm{for all } 0<t \leq t_0.
\eq
This follows from the bounds \eqref{ttr1} and \eqref{eq:nw_bd_exit_prob}.

(ii) We need to prove that for any $t>0$,
\begin{equation}\label{eq:nw_lim_x_0}
\lim_{x \dr 0} \lim_{n \rr \infty} nw_{nt}^{\th/n}(\sqrt{n}x) = \infty.
\end{equation}
Note that for any $x\leq t$ we have
\[
\aligned
 nw_{nt}^{\th/n}(\sqrt{n}x)
 &\geq    nw_{nx}^{\th/n}(\sqrt{n}x) \\
 &\geq   nP^{\th/n}(X \textrm{ survives  to generation } [ nx ] )\cdot P(W_{[ nx ]} \geq \sqrt{n}x)  \\
 &\geq   nP^{\th/n}(X \textrm{ survives  to generation } [ nx ] )(1-c_2e^{-c_3/x}),
\endaligned
\]
where we used Lemma \ref{Lemma-uni2}(ii) in the last inequality.
To bound the probability in the last term, note that by the weak convergence of $X$ to $\wt X$ we have
\[
\begin{aligned}
\liminf_n P^{\th/n}_n(X \textrm{ survives  to generation } [ nx ] )
&\geq P^{\th}_{\delta_0}(\wt X \textrm{ survives  to time } x ),
\end{aligned}
\]
which goes to 1 as $x\to 0$. Hence
\[
\liminf_{x\dr 0} \liminf_{n \rr \infty} nP^{\th/n}(X \textrm{ survives  to generation } [ nx ] )  = \infty.
\]
The conclusion follows.

(iii) Part (iii) follows from Lemma \ref{tight-n-w}(ii) and the monotonicity of $\phi$ in $t$.
\end{proof}

Next we prove uniqueness of positive solutions to (\ref{kpp}), with the initial and boundary conditions from Lemma \ref{lemma-boundary}. Note that this equation was presented in (\ref{kpp-uniq}).
 \begin{proposition} \label{prop-uniq}
There exists at most one positive solution $\phi$ to  equation~(\ref{kpp-uniq}).
\end{proposition}
\begin{proof}
Without loss of generality,  we set $\sigma_R= \sigma=1$ in (\ref{kpp-uniq}).
Define
$$
u(t,x) =e^{- \th t}  {\phi}(t,x -1), \q x >1, t> 0,
$$
which satisfies
\be \label{u-uniq}
\left\{
\begin{aligned}
\frac{\partial u}{\partial t}(t,x) &=\frac{1}{2} \frac{\partial^{2}u}{\partial x^{2} }(t,x) - {\frac{1}{2}} e^{\th t} u^{2}(t,x), \quad t>0, \ x>1, \\
 \lim_{t \dr 0} u (t,x) &= 0, \quad \mbox{for all } x>1, \\
\lim_{x \dr 1} u(t,x) &= \infty, \quad \mbox{for all }  t>0, \\
\lim_{x\rr \infty} \sup_{0\leq t\leq T} u(t,x) &= 0,\q {\mbox{    for all } T>0.}
\end{aligned}
\right.
\ee
Suppose that (\ref{u-uniq}) has two positive solutions $u_1$ and $u_2$ such that $u_1(t_0,x_0)\neq u_2(t_0,x_0)$ for some $x_0>1$ and $t_0\in (0,T)$. Without loss of generality, suppose $u_1(t_0,x_0)< u_2(t_0,x_0)$. Then by continuity, for $c>1$ that is {close enough to 1} we have
\begin{equation}\label{eq:u1u2}
u_1(t_0,x_0)< u_2(c^2 t_0, c x_0).
\end{equation}

Define
$$
v_2(t,x)=v_2(t,x;c) = c^2u_2(c^2 t,c x).
$$
Then $v_2$ satisfies
\be \label{v-2-uniq}
\left\{
\begin{aligned}
\frac{\partial v_2}{\partial t}(t,x) &=\frac{1}{2} \frac{\partial^{2}v_2}{\partial x^{2} }(t,x) - {\frac{1}{2}}e^{c^{2} \th t} v_2^{2}(t,x), \quad t>0, \ x>1, \\
 \lim_{t \dr 0} v_2 (t,x) &= 0, \quad \mbox{ for all }x>1, \\
 \lim_{x \dr 1} v_2(t,x) &<\infty, \mbox{ for all }t>0,\\
\lim_{x\rr \infty}  \sup_{t\in [0,T]} v_2(t,x) &= 0, {\mbox{    for all } T>0.}
\end{aligned}
\right.
\ee

Let
$$
f(t,x) =f(t,x;c) = u_1(t,x) - v_2(t,x).
$$
We have
\be \label{f-lim}
\lim_{x\dr 1} f(t,x) = \infty, \quad \textrm{for all } t>0.
\ee
By \eqref{eq:u1u2}, $\delta=\delta(c):= - \inf_{t,x} f(t,x;c) > 0.$
By (\ref{u-uniq}) and (\ref{v-2-uniq}), there exists~$M~>~0$ such that
$$
\sup_{x>M, \, 0< t\leq T} (u_1(t,x) +v_2(t,x)) < \delta/4,
$$
and therefore
\be \label{sup-f-M}
\inf_{x \geq M, \, 0< t\leq T} f(t,x)  > - \delta/2.
\ee
It follows that the infimum of $f(\cdot) = f(\cdot;c)$ must be attained at some point $(t^*,x^*)=((t^*(c),x^*(c)))\in (1,M) \times (0,T]$, and we have
\begin{equation}\label{eq:fxxft}
 \q \frac{\partial^{2} f}{\partial x^{2} }(t^*,x^*)\geq 0,\mbox{ and }\frac{\partial f}{\partial t} (t^*,x^*)\leq 0.
\end{equation}
However, by (\ref{u-uniq}) and (\ref{v-2-uniq}) we have
\begin{equation}\label{eq:fxx_ft_diff}
\frac{1}{2} \frac{\partial^{2}f}{\partial x^{2} }(t^*,x^*) - \frac{\partial f}{\partial t} (t^*,x^*)  = {\frac{1}{2}} e^{\th t^*}u_1^2(t^*,x^*)-{\frac{1}{2}}e^{ c^2\th t^*}v_2^2(t^*,x^*).
\end{equation}
When $\th\geq 0,$ the last term is negative due to that $0< u_1(t^*,x^*) < v_2(t^*,x^*)$. This contradicts~\eqref{eq:fxxft}, and we conclude the proof.

Consider now the case when $\th<0$. Take any sequence $(c_n)$ such that~\eqref{eq:u1u2} holds for all $c_n$ and $c_n\downarrow 1$. Note that by continuity we can choose $M$ independent of $c_n$ such that for all $n$ large enough,
$$
\sup_{x>M, \, 0< t\leq T} (u_1(t,x) +v_2(t,x;c_n)) < (u_2(t_0,x_0)- u_1(t_0,x_0))/4.
$$
Consider the point sequence $(t^*(c_n),x^*(c_n))\subseteq (1,M) \times (0,T]$, and suppose that $(t^o,x^o)$ is a limiting point. Note that
\[
\liminf_{n\rr \infty} \delta(c_n) \geq u_2(t_0, x_0) - u_1(t_0,x_0)>0.
\]
It follows that $(t^o,x^o)$ must be also inside $(1,M) \times (0,T]$, and by \eqref{eq:fxx_ft_diff}, when~$n$ is large enough,
\[
\frac{1}{2} \frac{\partial^{2}f}{\partial x^{2} }(t^*(c_n),x^*(c_n);c_n) - \frac{\partial f}{\partial t} (t^*(c_n),x^*(c_n);c_n) < 0,
\]
and we again get contradiction with \eqref{eq:fxxft}.
 \end{proof}

We are now ready to prove Theorem \ref{theorem-weak-con}.

\begin{proof}[Proof of Theorem \ref{theorem-weak-con}]
By Proposition \ref{prop-conv-sub}, Corollary \ref{corollary-pde}, Lemma \ref{lemma-boundary} and Proposition \ref{prop-uniq}, $(nw^{\theta/n}_{n t}(\sqrt{n}x)_{t>0,x>0}$ converges to $(\phi(t,x))_{t>0,x>0}$, which is the unique positive solution to (\ref{kpp-uniq}). Recall that $w^{\theta/n}_{\cdot}(\cdot)$ is for the case with a single initial particle, while $u^{\theta/n}_{\cdot}(\cdot)$ is for the case with $n$ initial particles. We have
\bd
u^{{\theta/n}}_{nt}(\sqrt{n}x) =1- \big(1-w^{{\theta/n}}_{nt}(\sqrt{n}x)\big)^n, \quad  \textrm{for all } x>0, \  t>0.
\ed
Therefore,
\be \label{u-phi}
\lim_{n \rr \infty} u^{{\theta/n}}_{nt}(\sqrt{n}x) =  1-e^{-\phi(t,x)}, \quad  \textrm{for all } x>0, \  t>0.
\ee

To finish the proof of Theorem \ref{theorem-weak-con}, in the below we analyze the exit probability of the limiting super-Brownian motion $\wt X$ with drift {$\theta$}, diffusion coefficient~$\sigma_R^{2}$ and branching coefficient $\sigma^2$.

For any $r>0$, choose a sequence of functions $\{\psi_{r,m}(x)\}_{r,m =1}^{\infty} \in C^{\infty}(\zz{R})$ satisfying the following:
\bn
\psi^{}_{r,m}(x)=
\left\{
\begin{array}{ll}
0,  \  \textrm{ if } -\infty  <x \leq r, \\\\
m,  \  \textrm{ if } r+\frac{1}{m} \leq x \leq m+r, \\\\
0,  \  \textrm{ if } x \geq m+r+1,
\end{array}
\right.
\en
and
\bn
0\leq \psi_{r,m}\leq m.
\en
Let $v_{r,m}$ be the solution to
\be \label{v-rm}
\frac{\partial v_{r,m}}{\partial t}(t,x) = \frac{\sigma_R^2}{2} \frac{\partial^{2}v_{r,m}}{\partial x^{2} }(t,x)+\theta v_{m,r}(t,x) -\frac{\sigma^2}{2} v_{r,m}^{2}(t,x)+ \psi_{r,m}(x), \ t>0, \ x\in \re,
\ee
with the initial condition $v_{r,m}(0,x) \equiv 0$.
By the same argument as for the convergence~(5) in \cite{pinsky95}, $v_{r,m}$ is increasing in $m$, and we can define
\be \label{v-r-lim}
v_r(t,x) = \lim_{m \rr \infty} v_{r,m}(t,x), \quad \textrm{for every } x\in\zz{R}, \, t>0.
\ee
Moreover, by repeating the argument for equation (6) in \cite{pinsky95}, with $\dl_{x}$ as the initial measure, we have
\be \label{v-r}
P^\theta_{\dl_x} \big( \wt L_{t}([r,\infty)) = 0  \big)  =e^{-  v_r(t,x)}.
\ee

We now analyze $v_r(t,x)$. By (\ref{v-rm}), (\ref{v-r-lim}) and the monotone convergence theorem and using further \eqref{v-r}, we get that that $v_r$ is a weak solution to
\be \label{v-r-eq}
\left\{
\aligned
\frac{\partial  v_r}{\partial t}(t,x) &= \frac{\sigma^2_R}{2} \frac{\partial^{2} v_r}{\partial x^{2} }(t,x)+\theta  v_r (t,x) -\frac{\sigma^2}{2} v_r^{2}(t,x),\\
&  \quad \mbox{for  } { t>0, \  -\infty < x <r, }\\
 v_r (0,x) &= 0, \quad -\infty <x <r, \\
\lim_{x \uparrow  r}  v_r(t,x) &= \infty, \quad t>0,\\
\lim_{x\rr -\infty}  \sup_{t\in [0,T]} v_r(t,x) &= 0, \mbox{    for all } T>0.
\endaligned
\right.
\ee
We want to strengthen the conclusion to be  that $v_r$ is a classical solution to~\eqref{v-r-eq}. To see this,
recall that $u_{r}$ is the minimal positive solution to~(\ref{sing-pde}). By Theorem A and Proposition A in~\cite{pinsky95},  we have
\be \label{exp-u}
P^\theta_{\dl_{x}} \big(\wt L_{t}(B_{r}(0)^{c}) =0 \big) = e^{-u_{r}{(t,x)}},
 \ee
and  $u_{r}$ satisfies that for every $\eps>0$, there exists $c_{\eps}>0$ and $M_{\eps}>0$ such that for all $ r>M_{\eps}$,  for all $t\geq 0 $ and $ x\in [0,r),$
\be \label{up-bound-u}
\begin{aligned}
u_{r}({t,x}) \leq \frac{1}{\sigma^2}\Big(\theta^++\frac{12\sigma_R^{2}r^2}{(r^2 - x^2)^2}\Big)  \exp\Big( -\Big(\frac{(r-x)^{2}}{2\sigma_R^{2}(1+\eps)t } &-  \theta^+ t -c_{\eps}\Big)_{+}\Big).
\end{aligned}
 \ee
Noting that $v_r(t,x) \leq u_r(t,x)$, using the bound (\ref{up-bound-u}) and regularity of weak solutions to parabolic PDE (see  Chapter 7.1.3 of \cite{Evans-book}), we see that $v_r$ is a positive classical solution to {(\ref{v-r-eq})}.

It follows that $\eta(t,x):=v_r(t,r-x), \ t>0, x>0$ solves \eqref{kpp-uniq}. However,  by Proposition \ref{prop-uniq}, the positive solution to (\ref{kpp-uniq}) is unique,  therefore
\be \label{phi-v}
\phi(t,x) = \eta(t,x)= v_r(t,r-x), \quad \textrm{for all } t>0, \ x>0.
\ee
Taking $r=x$ and using \eqref{v-r} we get the desired conclusion.
\end{proof}

Next, we prove Corollary \ref{corol-u-x}.
\begin{proof}[Proof of Corollary \ref{corol-u-x}]
First note that the convergence in (\ref{u-th-asymp}) has been derived in the proof of Theorem \ref{theorem-weak-con}.

Next we prove Part (i).
By Theorem \ref{theorem-weak-con}
and (\ref{exp-u}) we have
\[
\phi(t,x)\leq u_x(t,0).
\]
The bound in (i) then follows from (\ref{up-bound-u}).

To prove Part (ii), note that $w(t,x):= 2\th/\sigma^2 \cdot f_{\rho}((\rho \th t - \sqrt{\th} x/\sigma_R )_{+})$ satisfies
\bd
\left\{
\begin{aligned}
\frac{\partial w}{\partial t}(t,x) =&\frac{\sigma_R^2}{2} \frac{\partial^{2}w}{\partial x^{2} }(t,x)+ \th w(t,x) -  \frac{\sigma^2}{2}w^{2}(t,x), \quad t>0, \ x>0, \\
 \lim_{t \dr 0}w (t,x) =& 0, \\
\lim_{x\rr \infty}   w(t,x) =& 0 \quad \textrm{uniformly on } [0,T] \mbox{ for any } T>0.
\end{aligned}
\right.
\ed
Note that $\lim_{x \dr 0} w(t,x) <\infty$ for any $t>0$.

We want to show that  $w(t,x) \leq \phi(t,x)$ for all $x>0$ and $t\geq 0$. To do so, without loss of generality,  we set $\sigma_R= \sigma=1$ and define
$$
h(t,x) =e^{- \th t}  ({\phi}-w)(t,x-1), \q x >1, t> 0.
$$
The function $h(\cdot,\cdot)$ satisfies
\bd
\left\{
\begin{aligned}
\frac{\partial h}{\partial t}(t,x) =&\frac{1}{2} \frac{\partial^{2}h}{\partial x^{2} }(t,x)  - {\frac{1}{2}} e^{\th t}h(t,x)\cdot (\phi+w)(t,x), \quad t>0, \ x>1, \\
 \lim_{t \dr 0}h (t,x) =& 0, \mbox{ for all }  x>1, \\
\lim_{x \dr 1} h(t,x) =& \infty, \mbox{ for all } t>0,\\
\lim_{x\rr \infty}   h(t,x) =& 0 \quad \textrm{uniformly on } [0,T] \mbox{ for any } T>0.
\end{aligned}
\right.
\ed
Note that $(\phi+w)\geq 0$. By the same argument as in the proof of Proposition~\ref{prop-uniq}, we  get that $h(t,x)\geq 0$.

The conclusion follows.
\end{proof}

\section{{Proof} of Theorem \ref{theorem-speed}}\label{sec:travel_speed}

We first prove Part (i).

\begin{proof}[Proof of Theorem \ref{theorem-speed} (i)]
Recall that $\{L_k^{x}\}_{k \geq 0}$ stands for  the local time process of $X$.
By Corollary  A.2.7 in \cite{Lawler2010}, there exists $C>0$ such that for all $v>0$ and~$k\in\zz{N}$,
\begin{equation} \label{c52}
\begin{aligned}
&\q E_n^{\theta/n}\Big(  L_{n(k+1)}(B^{c}_{\sqrt{n}vk})\Big)\\
&= n \sum_{i\leq n(k+1)}\Big(1+\frac{\th}{n}\Big)^{i}P\big(|W_{i}|>\sqrt{n}vk\big) \\
&\leq   n\sum_{i\leq n(k+1)}\Big(1+\frac{\theta}{n}\Big)^{i} \cdot P\big(\max_{0 \leq i\leq n(k+1)} |W_{i}|>\sqrt{n}vk\big) \\
&\leq  n^2(k+1) \exp(\theta (k+1))\cdot \exp\left(-\frac{v^{2}k^2}{2\sigma^{2}_R(k+1)} + C\frac{v^{3}k^3}{\sigma_{R}^{3}(k+1)^2\sqrt{n}}\right)  \\
&= n^2(k+1)\exp\left((k+1)\Big(\th-\frac{v^{2}k^2}{2\sigma^{2}_R(k+1)^2} + \frac{Cv^{3} k^3}{\sigma_{R}^{3}(k+1)^3\sqrt{n}}\Big)\right).
\end{aligned}
\end{equation}
It follows that for every $v>\sqrt{2\th\sigma_{R}^{2}}$, there exists $N_{0}=N_{0}(v,\sigma_{R},\th)>0$ such that for all $ n>N_{0},$
\[
\sum_{k\in\zz{N} }E_n^{\theta/n}\Big(  L_{n(k+1)}(B^{c}_{\sqrt{n}vk})\Big) <\infty.
\]
Because $(L_\cdot(\cdot))$ is an integer-valued process,  by  the Borel-Cantelli lemma, we get that
\[
P_n^{\th/n}( L_{n(k+1)}(B^{c}_{\sqrt{n}vk}) = 0 \mbox{ for all } k \mbox{ large enough}) = 1.
\]
Note that for all $t>0$,
\[
  L_{nt}(B^{c}_{\sqrt{n}vt}) \leq L_{n([t]+1)}(B^{c}_{\sqrt{n}v[t]}).
\]
The conclusion  follows.
\end{proof}

Next we prove Part (ii).  The proof uses ideas from the proof of Theorem~2.1 in \cite{Kyprianou05}. Recall that $X_k(x)$ is the number of particles at site $x$ at generation $k$. For any $\gamma\in \re$, define
\be \label{m-beta}
m(\gamma)=m^{(n)}(\gamma) = \Big(1+\frac{\theta}{n}\Big)\sum_{z\in \ze}a_{z}e^{-\gamma z},
\ee
and
$$
W^{\gamma}_{k}= m^{-k}(\gamma) \sum_{z\in \ze}e^{-\gamma z}X_{k}(z), \quad k=0,1,...
$$
Then $\{W^{\gamma}_{k}\}_{k\geq 0}$ is a martingale; see Chapter {VI.4} of \cite{Athreya-1972} or Theorem 1 in \cite{Kingman}. Because $\{W^{\gamma}_{k}\}_{k\geq0}$ is nonnegative, the limit  $W^{\gamma} := \lim_{k\rr\infty }W_{k}^{\beta}$ almost surely exists, and by Fatou's lemma,
 $E_1^{\theta/n}(W^{\gamma})\leq 1$. The following lemma characterizes when $E_1^{\theta/n}(W^{\gamma})~=~1$ and provides the key ingredient in proving Part (ii) of Theorem~\ref{theorem-speed}.
\begin{lemma} \label{lem-exp-w} If $|\beta| < \sqrt{{2\theta}/{\sigma^2_R}}$, then for all sufficiently large $n$,
$$
E_1^{\theta/n}(W^{\beta/\sqrt{n}}) = 1.
$$
\end{lemma}
\begin{proof}
The proof is based on  Lemma 5 in \cite{Biggins77} {(see also Theorem 3.2 in \cite{Shi-SF}).}
From our assumptions on the step distribution and branching law, by Lemma~5 in \cite{Biggins77}, we need to verify that if $|\beta| < \sqrt{{2\theta}/{\sigma^2_R}}$, then for all sufficiently large $n$,
\be \label{mb1}
m(\beta/\sqrt{n}) \exp\Big(-\frac{\beta}{\sqrt{n}} \frac{m'(\beta/\sqrt{n})}{m(\beta/\sqrt{n})}\Big) >1.
\ee
Note that
\bd
-\frac{m'(\beta/\sqrt{n})}{m(\beta/\sqrt{n})} = \frac{(\beta/\sqrt{n}) \sum_{z\in \ze}a_z z\,e^{-(\beta/\sqrt{n}) z} }{\sum_{z\in \ze}a_z e^{-(\beta/\sqrt{n})z}}.
\ed
By the Taylor expansion and using the fact that $\sum_{z \in \ze} z a_z =0$ we get
$$
 \sum_{z\in \ze}a_z z\,e^{-(\beta/\sqrt{n}) z} = -\sigma^2_R\frac{\beta}{\sqrt{n}} +O(n^{-1}),
$$
and
\be\label{mb1.5}
 \sum_{z\in \ze}a_ze^{-(\beta/\sqrt{n}) z} = 1+\frac{\sigma^2_R}{2}\frac{\beta^2}{n} +O(n^{-3/2}).
\ee
It follows that
\be \label{mb2}
\exp\Big(-\frac{\beta}{\sqrt{n}} \frac{m'(\beta/\sqrt{n})}{m(\beta/\sqrt{n})}\Big) =1 - \sigma^2_R \frac{\beta^2}{n} +O(n^{-3/2}).
\ee
From (\ref{m-beta}), (\ref{mb1.5}) and (\ref{mb2}) we get
\bn
&& m(\beta/\sqrt{n}) \exp\Big(-\frac{\beta}{\sqrt{n}} \frac{m'(\beta/\sqrt{n})}{m(\beta/\sqrt{n})}\Big)\\
& =& \Big(1+\frac{\theta}{n}\Big)\Big(1+\frac{\sigma^2_R}{2}\frac{\beta^2}{n}\Big) \Big(1 - \sigma^2_R \frac{\beta^2}{n}  \Big) +O(n^{-3/2})\\
&=& 1 + \Big(\th - \frac{\sigma_R^2\beta^2}{2}\Big)\frac{1}{n} +O(n^{-3/2}),
\en
and we verify (\ref{mb1}).
\end{proof}

We are now ready to prove Theorem~\ref{theorem-speed} (ii).

\begin{proof}[Proof of Theorem \ref{theorem-speed} (ii)]
Because the branching random walk is homogeneous, the probability that the martingale limit $W^\gamma$ is positive is \emph{independent} of the start position. Hence, by standard Galton-Watson arguments, similar to the proof of {Lemma 2.2 in ~\cite{Shi-SF}}, for any $\gamma\in\zz{R}$, $P_1^{\theta/n}(W^{\gamma} > 0) $ is either $0$ or equal to the survival probability of $X^{}$. Therefore by Lemma \ref{lem-exp-w}, if $|\beta| < \sqrt{2\theta/\sigma^2_R}$, then for all sufficiently large~$n$,
\be \label{w-sur}
P_1^{\theta/n}(W^{-\beta/\sqrt{n}} >0) = P_1^{\theta/n}(X \textrm{ survives}).
\ee
Define
$$
\wt W_{t}^{-\beta/\sqrt{n}} = W^{-\beta/\sqrt{n}}_{[ nt ]} = (m(-\beta/\sqrt{n}))^{-[ nt ]} \sum_{z\in \ze}e^{(\beta/\sqrt{n}) z}X_{nt}(z), \quad t \geq 0,
$$
and
\begin{equation}\label{eq:tW_W}
\wt W^{-\beta/\sqrt{n}}:=\lim_{t\to\infty} \wt W_{t}^{-\beta/\sqrt{n}} =  W^{-\beta/\sqrt{n}}.
\end{equation}
Now pick any $0<\beta < \sqrt{2\theta/\sigma^2_R}$, and let $\eps\in (0,\beta)$ be an arbitrarily small number. Note that
$$
e^{(\beta/\sqrt{n}) z}  \mathds{1}_{\{z\leq \sqrt{n}(\beta-\eps)\sigma_R^2\,t\}} \leq e^{(\beta-\eps)z/\sqrt{n}}e^{\eps(\beta-\eps)\sigma_R^2\,t}.
$$
It follows from the Taylor expansion  in the proof of Lemma \ref{lem-exp-w} that for all $n$ large enough,
\bn
&&(m(-\beta/\sqrt{n}))^{-[nt]} \sum_{z\in \ze}e^{(\beta/\sqrt{n}) z}\mathds{1}_{\{z\leq \sqrt{n}(\beta-\eps)\sigma_R^2\,t\}}X_{nt}(z) \\
&=& e^{-(\theta +\frac{1}{2}\beta^2\sigma_R^2+O(n^{-1/2}))\,t} \sum_{z\in \ze}e^{(\beta/\sqrt{n})z} \mathds{1}_{\{z\leq \sqrt{n}(\beta-\eps)\sigma_R^2\,t\}}X_{nt}(z) \\
&\leq&  \wt W_t^{-(\beta-\eps)/\sqrt{n}} e^{-(\frac{1}{2}\eps^2\sigma_R^{2}+O(n^{-1/2}))\,t}.
\en
Because $\wt W_t^{-(\beta-\eps)/\sqrt{n}}$ converges almost surely, the last term converges to 0 as $t\to\infty$, and
we obtain that
\bn
\lim_{t\rr \infty}
m(-\beta/\sqrt{n}))^{-nt} \sum_{z\in \ze} e^{(\beta/\sqrt{n}) z}\mathds{1}_{\{z> \sqrt{n}(\beta-\eps)\sigma_R^2t\}}X_{nt}(z)= \wt W^{-\beta/\sqrt{n}}.
\en
It follows from \eqref{eq:tW_W} and \eqref{w-sur} that
$$
P_1^{\theta/n}(X_{nt}( (\sqrt{n}(\beta-\eps)\sigma_R^2\, t,\infty) )> 0  \mbox{ for all } t \mbox{ large enough} \, |\, X \textrm{ survives}) = 1.
$$
Because $\beta$ can be arbitrarily close to $\sqrt{2\theta/\sigma^2_R}$ and $\eps$ can be arbitrarily small, we get the desired conclusion.
\end{proof}

\section{Proof of Theorem \ref{theorem-u-inf}}\label{sec:Minfty}

In this section, we study {$M$}, the maximum displacement throughout the whole process. Denote
$$
w_\infty^{\th/n}(x)=P_1^{\th/n}({M} \geq x), \quad x \geq 0.
$$
In the following lemma, we show that for any $\theta \in \re$ {and $x_0>0$}, the function sequence $(n w_\infty^{\th/n}(\sqrt{n} x))_{x\geq x_0}$  uniformly converges as $n\rr \infty.$
Recall that the convergence of $nw^{\theta/n}_{n t}(\sqrt{n}x)$ to $\phi(t,x)$  was established in the proof of Theorem~\ref{theorem-weak-con}.


\begin{lemma}\label{lemma:conv_Minfty}
For any $\theta\in\zz{R}$ and $x_0>0,$ $\big(n w_\infty^{\th/n}(\sqrt{n} x)\big)_{x\geq x_0}$  uniformly converges to $\big(\psi(x):=\lim_{t\to\infty} \phi(t,x)\big)_{x\geq x_0}.$
\end{lemma}
\begin{proof}
By duality between the supercritical and subcritical branching processes (see discussion before Lemma \ref{tight-n-w}), we need only to show the convergence when~$\theta \leq 0.$ For any $t>0$, we have
\begin{equation}\label{eq:wnt_winf}
\aligned
n w_{nt}^{\th/n}(\sqrt{n} x)
&\leq n w_\infty^{\th/n}(\sqrt{n} x)\\
&\leq n w_{nt}^{\th/n}(\sqrt{n} x) + n P_1^{\th/n}(X \mbox{ survives to generation } [nt])\\
&= n w_{nt}^{\th/n}(\sqrt{n} x) + O(1/t).
\endaligned
\end{equation}
It  follows  that $n w_\infty^{\th/n}(\sqrt{n} x)$ pointwise converges to $\psi(x)$, and by Theorem~\ref{theorem-weak-con} and Proposition \ref{prop-conv-sub}, the limiting function $\psi(\cdot)$ is finite and continuous. Moreover, by Dini's Theorem, the convergence of $\phi(t,x)\to \psi(x)$ is uniform over any compact interval inside $(0,\infty)$. By~\eqref{eq:wnt_winf} again, $n w_\infty^{\th/n}(\sqrt{n} x)$  converges to~$\psi(x)$ uniformly over any compact interval.  Note further that $ n w_{\infty}^{\th/n}(\sqrt{n} x)\leq  n w_{\infty}(\sqrt{n} x)$, which,  by Theorem 1 in \cite{LS15}, is $ O(1/x^2)$, so the  uniform convergence over $[x_0,\infty)$ follows.
\end{proof}

In the following proposition, we describe the limiting function $\psi(\cdot).$

For any $\th > 0$, define
\[
\wt{w_\infty^{\th/n}}(\sqrt{n} x)=P_1^{\th/n}(M_\infty\geq \sqrt{n} x\,|\, \mbox{extinction}).
\]
By duality between the supercritical and subcritical branching processes we have  $(n\wt{w_\infty^{\th/n}}(\sqrt{n} x))$ converges to $(\wt{\psi}(x))$, which,  by (\ref{b222}), satisfies the following  relationship with $(\psi(x))$:
\begin{equation}\label{eq:phi_psi}
  \psi(x) = \wt{\psi}(x) + 2\theta/\sigma^2,\q\mbox{ for all } x >0.
\end{equation}
Recall that $\bar \tau^{Y}_{x}$ was defined in (\ref{tau-B}) for any continuous process $\{Y_{t}\}_{t\geq 0}$ and $x\in \re$. We further define for $x<y \leq \infty$,
\[
 \tau^{Y}_{y}  =\min\{t\geq 0: Y_{t} \geq y\}, \mbox{ and }\bar \tau^{Y}_{x,y}=\bar \tau^{Y}_{x} \wedge  \tau^{Y}_{y}.
\]

\begin{prop}\label{prop:psi}
The limiting function $\psi(\cdot)$ in Lemma \ref{lemma:conv_Minfty} satisfies that
\begin{equation}\label{eq:psi}
\left\{
\aligned
\frac{\sigma^2_R}{2} \frac{\partial^{2}\psi}{\partial x^{2} }&= - \theta  \psi +\frac{\sigma^2}{2}\psi^2 ,  \quad x >0, \\
\lim_{x \to 0+} \psi(x) &= \infty,\\
\lim_{x \to \infty} \psi(x) &=\frac{2\th^+}{\sigma^2}.
\endaligned
\right.
\end{equation}
\end{prop}
\begin{proof}
{The fact that} $\psi(\cdot)$ satisfies the boundary conditions in \eqref{eq:psi} follows from \eqref{eq:wnt_winf}, \eqref{eq:phi_psi}, Lemma \ref{lemma-boundary} and Theorem 1 in \cite{LS15}.
It remains to show that $\psi(\cdot)$ satisfies the ODE in \eqref{eq:psi}.
By~\eqref{eq:phi_psi} again, it suffices to show the case when $\th\leq0$.

We only give the sketch of proof because it uses similar techniques to proving the convergence of $(n w_{nt}^{\th/n}(\sqrt{n} x))$. A simple modification of the proof of  Lemma~4.5 in \cite{NZ17} yields the following discrete Feynman-Kac formula for~$w_\infty^{\th/n}(x)$: for all $0<x_0< {x}<y\leq \infty$,
\begin{equation*}\label{eq:FK_winfty}
w_\infty^{\th/n}(x)=E_{x}^{\th/n}\Big((1+\th/n)^{\bar \tau_{x_0,y}} w_\infty^{\th/n}(\W_{\bar \tau_{x_0,y}})\prod_{j=1}^{\bar \tau_{x_0,y}  }\big[1-H^{\th/n}( w_\infty^{\th/n}\big(\W_{j})\big)\big] \Big).
\end{equation*}
It follows by a similar argument to the proofs of Proposition \ref{prop-phi-con} and Corollary~\ref{corollary-pde} that the limiting function $\psi(\cdot)$ satisfies that  for all $0<x_0<y\leq \infty$,
\begin{equation}\label{eq:psi-int}
\psi(x) =  E_{x }\left(\exp\left(\th \bar \tau_{x_{0},y}^{{\sigma_R}B}-\frac{ \sigma^{2}}{2} \int_{0}^{\bar \tau_{x_{0},y}^{{\sigma_R}B}}\psi({\sigma_R}B_{s})ds\right)\cdot \psi(\sigma_R B_{\bar \tau_{x_{0},y}^{{\sigma_R}B}})\right).
\end{equation}
Based on this expression, we want to show that $\psi(\cdot)$ satisfies
\[
\frac{\sigma^2_R}{2} \frac{\partial^{2}\psi}{\partial x^{2} }= - \theta  \psi +\frac{\sigma^2}{2}\psi^2 ,  \quad x_0< x< \infty.
\]
By Theorem 3.1 in {Chapter 6} of \cite{Dynkin02}, we only need to verify that $\psi(\cdot)$ is Lipschitz. To do so, we take $y=\infty$ in \eqref{eq:psi-int}, which yields
\[
\psi(x) = \psi(x_0) E_{x }\left(\exp\left(\th \bar \tau_{x_{0}}^{{\sigma_R}B}-\frac{ \sigma^{2}}{2} \int_{0}^{\bar \tau_{x_{0}}^{{\sigma_R}B}}\psi({\sigma_R}B_{s})ds\right)\right).
\]
By the strong Markov property, for any $\delta\geq 0$ and $x> x_0,$
\[
\aligned
\psi(x+\delta)
&=\psi(x_0) E_{x+\dl }\left(E_{x+\dl }\left(\exp\left(\left.\th \bar \tau_{x_{0}}^{{\sigma_R}B}-\frac{ \sigma^{2}}{2} \int_{0}^{\bar \tau_{x_{0}}^{{\sigma_R}B}}\psi({\sigma_R}B_{s})ds\right)\right| \mathcal{F}_{\bar \tau_{x}^{{\sigma_R}B}}\right)\right)\\
&=\psi(x) E_{x+\delta}\left(\exp\left(\th \bar \tau_{x}^{{\sigma_R}B}-\frac{ \sigma^{2}}{2} \int_{0}^{\bar \tau_{x}^{{\sigma_R}B}}\psi({\sigma_R}B_{s})ds\right)\right).
\endaligned
\]
It follows that $\psi(x)$ is decreasing in $x$, and we have
\[
\begin{aligned}
0&\leq \psi(x) - \psi(x+\delta)\\
 &=\psi(x) \left( 1 -  E_{x+\delta}\left(\exp\left(\th \bar \tau_{x}^{{\sigma_R}B}-\frac{ \sigma^{2}}{2} \int_{0}^{\bar \tau_{x}^{{\sigma_R}B}}\psi({\sigma_R}B_{s})ds\right)\right) \right)\\
 &\leq \psi(x) \left( 1 - E_{x+\delta}\left(\exp\left((\th - \frac{ \sigma^{2}}{2} \psi(x)) \bar \tau_{x}^{{\sigma_R}B}\right)\right)\right)\\
 &=\psi(x) \left( 1 - \exp\left(-\sqrt{-2\th + \sigma^{2}\psi(x)}\cdot \delta/\sigma_R \right)\right)\\
  &=\psi(x)\sqrt{-2\th + \sigma^2 \psi(x)}\cdot O(\delta).
\end{aligned}
\]
Therefore, $\psi(\cdot)$ is Lipschitz and we complete the proof.

\end{proof}

In the rest of this section we assume  that $\th > 0$.

We are interested in the asymptotic behavior of $\psi(x)$ as $x\to \infty$. By \eqref{eq:phi_psi}, it suffices to study the asymptotic behavior of $\wt{\psi}(x)$.
For notational ease, denote
\begin{equation}\label{eq:a_b}
a = \frac{2 \theta }{\sigma_R^2} \q\mbox{and}\q b = \frac{\sigma^2}{\sigma_R^2}.
\end{equation}
Then $\wt{\psi}(\cdot)$ satisfies
\begin{equation}\label{eq:Minfty_sub}
\left\{
\aligned
\frac{\partial^{2}\wt{\psi}}{\partial x^{2} } &=  a\wt{\psi} + b\wt{\psi}^2.\\
\lim_{x \to 0+} \wt{\psi}(x) &= \infty,\\
\lim_{x \to \infty} \wt{\psi}(x) &=0.
\endaligned
\right.
\end{equation}
In the following lemma, we establish uniqueness of solutions to (\ref{eq:Minfty_sub}) (note that the uniqueness does not {follow} from Theorem 3.1 in \cite{Dynkin02}, which applies to the case with bounded domain and given boundary condition).
\begin{lemma}\label{lemma:uniq_psi}
For any $a> 0$ and $ b >0$, there exists at most one solution to~\eqref{eq:Minfty_sub}.
\end{lemma}
\begin{proof}
The proof is similar to that of Proposition \ref{prop-uniq}. Suppose that there exist two solutions $\wt{\psi}_i,\ i = 1, 2,$ to \eqref{eq:Minfty_sub}, and suppose that $\wt{\psi}_1(x_0) < \wt{\psi}_2(x_0)$ for some $x_0>0$. Take $c>1$ close enough to $1$ so that
\[
\wt{\psi}_1(x_0) < \wt{\psi}_2(c x_0 + c-1).
\]
Define $\eta(x) = \wt{\psi}_2(c x + c-1)$ for all $x>0$, which satisfies that
\begin{equation}\label{eq:eta}
\left\{
\aligned
\frac{\partial^{2}\eta }{\partial x^{2} } &=  a c^2 \eta + bc^2\eta^2.\\
\lim_{x \to 0+} \eta(x) &< \infty,\\
\lim_{x \to \infty} \eta(x) &=0.
\endaligned
\right.
\end{equation}
Let $f(x) = \wt{\psi}_1(x) - \eta(x)$ for $x>0$. Then $f(\cdot)$ must attain its minimum at some point $x^*\in(0,\infty)$, and we have
\[
 f(x^*)=\wt{\psi}_1(x^*) - \eta(x^*) <0,\q\mbox{and}\q \frac{\partial^{2} f }{\partial x^{2} } (x^*) \geq 0.
\]
However, by \eqref{eq:Minfty_sub} and \eqref{eq:eta},
\[
\frac{\partial^{2} f }{\partial x^{2} } (x^*) = a\wt{\psi}_1(x^*) -  a c^2 \eta(x^*) +  b\wt{\psi}_1^2(x^*) -  bc^2\eta^2(x^*) < 0,
\]
which is a contradiction.
\end{proof}

We are now ready to prove Theorem \ref{theorem-u-inf}.

\begin{proof}[Proof of  Theorem \ref{theorem-u-inf}.]
The convergence \eqref{eq:conv_M}  and that $\psi(\cdot)$ satisfies \eqref{psi-eq} follow from Lemma~\ref{lemma:conv_Minfty}, Proposition \ref{prop:psi} and Lemma \ref{lemma:uniq_psi}. It remains to show \eqref{eq:psi_formula} and \eqref{psi-asymptotic}.
By \eqref{eq:phi_psi}, it suffices to study $\wt{\psi}(\cdot),$ which satisfies equation \eqref{eq:Minfty_sub}.

Note that \eqref{eq:Minfty_sub} is a second-order autonomous ODE, and it admits the following positive solution given by an implicit function:
\begin{equation} \label{eq:impl}
-\int \frac{1}{\sqrt{a \wt{\psi}^2 + (2b/3)\ \wt{\psi}^3+C_1}}\, d \wt{\psi} = x + C_2,
\end{equation}
where $C_1$ and $C_2$ are two constants.
Letting $x\to \infty$ and noting that\\ $\lim_{x \to \infty} \wt{\psi}(x) = 0$, we conclude  that $C_1=0.$
When $C_1=0$, the left hand side of \eqref{eq:impl} can be explicitly integrated out, and we obtain
\[
\frac{2}{\sqrt{a}}\arcoth\left(\sqrt{2b/(3a)  \wt{\psi} + 1}\,\right) =  x+ C_2,
\]
where for any $s> 1$,
\[
  \arcoth(s) = \frac{1}{2}\log((s+1)/(s-1)).
\]
Letting $x\to 0 +$ and noting that $\lim_{x \to 0+} \wt{\psi}(x) = \infty$, we see  that $C_2=0$ and so
\begin{equation}\label{eq:explicit_soln}
\frac{2}{\sqrt{a}}\arcoth\left(\sqrt{2b/(3a)  \wt{\psi} + 1}\,\right) =  x.
\end{equation}
Therefore,
$$
\wt \psi(x) = \frac{3a}{2b} \left( \coth^2\left(\frac{\sqrt{a}x}{2}\right)-1 \right),
$$
and so
\[
\wt{\psi}(x) \sim \frac{6a}{b} \exp(-\sqrt{a}\, x),\q\mbox{as } x\to \infty.
\]
Plugging $a$ and $b$ in \eqref{eq:a_b} yields
\begin{equation*}
  \wt{\psi}(x) \sim \frac{12 \theta}{\sigma^2} \exp\left(-\sqrt{\frac{2 \theta }{\sigma_R^2}}\, x\right),\q\mbox{as } x\to \infty.
\end{equation*}
The conclusions \eqref{eq:psi_formula} and \eqref{psi-asymptotic} follow.
\end{proof}

\bigskip

\section*{Acknowledgements}
We are very grateful to an anonymous referee for  careful reading of the manuscript,
and for a number of useful comments and suggestions that significantly improved this paper.
\\
\\
\textbf{Availability of data and material.} Not applicable.
\medskip \\
\textbf{ Compliance with ethical standards.} The authors have no conflicts of interest to declare that are relevant to the content of this article.
\medskip \\
\textbf{Code availability.} Not applicable.

\bibliographystyle{plain}
\printindex

\bigskip
\noindent Eyal Neuman: Department of Mathematics, Imperial College London, London, SW7 1NE, U.K. e.neumann@imperial.ac.uk

\medskip
\noindent Xinghua Zheng: Department of Information Systems
 Business Statistics and Operations Management, Hong Kong University of Science and
Technology, Clear Water Bay, Kowloon, Hong Kong. xhzheng@ust.hk

\end{document}